\documentclass[12pt,a4paper]{article}

\usepackage{amssymb}
\usepackage{tikz-cd}
\usepackage{setspace}
\usepackage{amsmath}
\usepackage{hyperref}
\usepackage[shortlabels]{enumitem}
\usepackage{mathtools}
\usepackage{xcolor}
\usepackage{amsthm}
\usepackage{tocbasic}
\DeclareTOCStyleEntry[
  beforeskip=.3em plus 3pt,
  pagenumberformat=\textbf
]{tocline}{section}
\usepackage[matrix,arrow]{xy}
\usepackage[margin=1in]{geometry}

\DeclareMathAlphabet\mathbfcal{OMS}{cmsy}{b}{n}

\numberwithin{equation}{section}

\newtheorem{theorem}{Theorem}[section]
\newtheorem{definition}[theorem]{Definition}

\newtheorem{lemma}[theorem]{Lemma}

\newtheorem{proposition}[theorem]{Proposition}
\newtheorem{corollary}[theorem]{Corollary}
\newtheorem{remark}[theorem]{Remark}

\newtheorem{example}[theorem]{Example}

\newtheoremstyle{named}{}{}{\itshape}{}{\bfseries}{.}{.5em}{#1 \thmnote{#3}}
\theoremstyle{named}
\newtheorem*{namedtheorem}{Property}

\newcommand{\PP}{{\mathbb P}} \newcommand{\RR}{\mathbb{R}}
\newcommand{\QQ}{\mathbb{Q}} \newcommand{\CC}{{\mathbb C}}
\newcommand{\ZZ}{{\mathbb Z}} \newcommand{\NN}{{\mathbb N}}

\newcommand{\cX}{\mathcal{X}}

\newcommand{\cI}{\mathcal{I}}
\newcommand{\cJ}{\mathcal{J}}

\newcommand{\cT}{\mathcal{T}}

\newcommand{\vol}{\operatorname{vol}}

\newcommand{\ddc}{\mathrm{dd}^\mathrm{c}}
\newcommand{\PSH}{\mathrm{PSH}}

\DeclareMathOperator{\supp}{supp}
\DeclareMathOperator{\MA}{MA}

\title{Transcendental Okounkov bodies}

\usepackage{hyperref}
\hypersetup{colorlinks,linkcolor={black},citecolor={black}}

\begin{document}

\author{T. Darvas, R. Reboulet, D. Witt Nystr\"om, M. Xia, K. Zhang \vspace{0.2cm}}


\date{\small{\emph{To the memory of Jean-Pierre Demailly}}}

\maketitle

\begin{abstract}
We show that the volume of transcendental big $(1,1)$-classes on compact K\"ahler manifolds can be realized by convex bodies, thus answering questions of Lazarsfeld--Musta\c{t}\u{a} and Deng. In our approach we use an approximation process by partial Okounkov bodies together with properties of the restricted volume, and we study the extension of K\"ahler currents, as well as the bimeromorphic behavior of currents with analytic singularities. We also establish a connection between transcendental Okounkov bodies and toric degenerations.
\end{abstract}

\tableofcontents

\section{Introduction}

In toric geometry, there is a striking correspondence between Delzant polytopes $\Delta$ and polarized toric manifolds $(X_{\Delta},L_{\Delta})$. Much of the geometry of the toric manifold is encoded in its Delzant polytope: for instance, the volume of $L_{\Delta}$, defined as 
\[
\vol(L_{\Delta})\coloneqq \lim_{k\to \infty}\frac{h^0(X_{\Delta},kL_{\Delta})}{k^n/n!},
\]
is seen to equal $n!$ times the Euclidean volume of $\Delta$. 

Toric manifolds are very special, due to their large group of symmetries. Thus, it was quite remarkable that Okounkov in the 90s \cite{Oko96} showed how one can naturally associate a convex body $\Delta(L)$ to any polarized manifold $(X,L)$. These bodies are now known as Okounkov bodies (or Newton--Okounkov bodies).

To construct the Okounkov body, one fixes a  complete flag of submanifolds $Y_{\bullet}=(X=Y_0\supset Y_1\supset \dots \supset Y_{n-1}\supset Y_n )$. To any $s\in H^0(X,kL)\setminus \{0\}$ we associate a valuation vector $\nu(s)\coloneqq (\nu_1,\dots,\nu_n)\in \RR^n$ in the following way. First, we let $\nu_1$ be the order of vanishing of $s$ along $Y_1$. Next, if $s_{Y_1}$ is a defining section for $Y_1$, we get that  $(s/s_{Y_1}^{\nu_1})|_{Y_1}$ is a non-zero section of $L^k \otimes \mathcal O(-\nu_1 Y_1)$ restricted to $Y_1$. We set $\nu_2$ to be the order of vanishing along $Y_2$ of this section. Continuing this process yields the full valuation vector $\nu(s)$. The Okounkov body of $L$, denoted $\Delta_{Y_{\bullet}}(L)$, is then defined as the closure of the set of rescaled valuation vectors 
\[
\Delta_{Y_\bullet}(L)\coloneqq \overline{\left\{\nu(s)/k: s\in H^0(X,kL)\setminus \{0\}, k\in \NN\right\}}.
\]
Convexity of the Okounkov body follows from the fact that $\nu(s \otimes t)=\nu(s)+\nu(t)$. A more difficult fact, however, is that the Okounkov body satisfies the volume identity: 
\begin{equation} \label{eqn:voleq1}
\vol_{\RR^n}(\Delta_{Y_\bullet}(L))=\frac{1}{n!}\vol(L).
\end{equation}

Later Lazarsfeld--Musta\c{t}\u{a} \cite{LM09} and Kaveh--Khovanskii \cite{KK12} showed that Okounkov's construction works for more general big line bundles as well, together with the analogous volume identity (\ref{eqn:voleq1}). One should note that the volume of big line bundles is a much more subtle notion compared to the ample case; using convex bodies to study this object was indeed a breakthrough. For example, Lazarsfeld--Musta\c{t}\u{a} used Okounkov bodies to show that the volume function is $C^1$ on the cone of big $\RR$-divisors \cite[Corollary C]{LM09}. Since then, numerous applications of Okounkov bodies have been found in algebraic and complex geometry (see e.g. \cite{BC11,BKMS15,LX18,BJ20,JL23}, to only name a few works).

The Okounkov body is a numerical invariant of the line bundle, i.e., it only depends on the first Chern class $c_1(L)$ of $L$, which is a real cohomology class of bidegree $(1,1)$.  
Real cohomology classes that do not correspond to line bundles are called transcendental.
The weaker positivity notions of being big or pseudoeffective have been extended to transcendental classes by Demailly, and Boucksom showed in \cite{Bou02b} how to define the volume $\vol(\xi)$ of a transcendental class $\xi$.

In light of the above, given a flag $Y_\bullet$, a very natural question  raised by Lazarsfeld--Musta\c{t}\u{a} \cite[p. 831]{LM09} asks whether one can define the Okounkov body $\Delta_{Y_\bullet}(\xi)$ associated with a transcendental class $\xi$, in such a way that it satisfies the volume identity
\begin{equation}
    \label{eq:vol-eq}
    \vol_{\RR^n}(\Delta_{Y_\bullet}(\xi))=\frac{1}{n!}\vol(\xi).
\end{equation} 

There are some notable open problems related to the volume of big transcendental classes, such as obtaining transcendental holomorphic Morse inequalities, or the continuous differentiability of the volume on the big cone \cite{BDPP13}. Identity \eqref{eq:vol-eq} could turn out to be useful with tackling these problems, the same way as it was in the line bundle case \cite[Corollary C]{LM09}. 

As we explain now, in \cite{Den17} Deng proposed a construction of Okounkov bodies $\Delta(\xi)$ for a transcendental class $\xi$. In the transcendental case global sections make no sense, but these classes admit a lot of currents. The vanishing order of a global section $s$ of a line bundle can be expressed using Lelong numbers of the current of integration $[s =0]$. Deng's construction rests on the idea that computing Lelong numbers of appropriate currents should be the transcendental analogue of computing vanishing order of sections in the algebraic case.

Let $(X,\omega)$ be a connected compact K\"ahler manifold of dimension $n$. Let $\xi\coloneqq \{\theta\}\in H^{1,1}(X,\RR)$ be a big (or pseudoeffective) cohomology class, with $\theta$ being a smooth closed $(1,1)$-form. Assume that $X$ has a complete flag of submanifolds 
\[
Y_\bullet\coloneqq  (X=Y_0\supset Y_1\supset \dots \supset Y_{n-1}\supset Y_n ),
\] 
 Let $\mathcal{A}(X,\xi)$ be the set of positive currents $\theta + \ddc u$ in $\xi$  with analytic singularities. This means that locally $u = c \log (\sum_j |f_j|^2) + h$, for $c \in \mathbb Q_+$, $f_j$ holomorphic, and $h$ bounded.

To any $T_0 \in\mathcal{A}(X,\xi) $ one associates a valuation vector $\nu(T_0) = (\nu_1, \ldots, \nu_n) \in \mathbb R^n$ in the following manner. Let $\nu_1$ equal $\nu(T_0,Y_1)$, the generic Lelong number of $T_0$ along $Y_1$. One can show that $T_1 \coloneqq (T_0 - \nu_1 [Y_1])|_{Y_1}$ has analytic singularities on $Y_1$. Then one sets $\nu_2 \coloneqq \nu(T_1, Y_2)$ and continues the above process until all components of $\nu(T_0)$ are defined. The Okounkov body of $\xi$, denoted $\Delta_{Y_{\bullet}}(\xi)$ (or just $\Delta(\xi)$, once the flag is fixed), is then defined as:
\[\Delta(\xi) \coloneqq  \overline{\nu(\mathcal{A}(X,\xi))} \subset \mathbb R^n.\]
As in the line bundle case, the convexity of the Okounkov body follows from the fact that $\nu(S+T)=\nu(S)+\nu(T)$.
When $\xi$ is a pseudoeffective class, one puts $
\Delta(\xi)=\cap_{\varepsilon>0}\Delta(\xi+\varepsilon\{\omega\})$. For more details, the reader is invited to consult \S~\ref{sec:construct-body}.

Deng's body coincides with Okounkov's when $\xi$ is the first Chern class of a big line bundle \cite{Den17}. Thus, this  body is a natural generalization of the Okounkov body to the transcendental setting, and we will work with Deng's definition in our paper.

It remained to show that the transcendental Okounkov body satisfies (\ref{eq:vol-eq}). Deng proved that this is true when the dimension of $X$ is at most two, but the general case was open until now.
Our main result proves the volume identity in full generality:

\begin{theorem}\label{thm:main}
Let $\xi$ be a big cohomology class on $X$ and $Y_\bullet$ be a flag on $X$. Then the associated Okounkov body $\Delta(\xi)$ of $\xi$ satisfies the volume equality
\begin{equation}\label{eq: main_volume_eq}
\vol_{\RR^n}(\Delta(\xi))=\frac{1}{n!}\vol(\xi).
\end{equation}
\end{theorem}

Passing to the limit, the volume identity can be extended to pseudoeffective classes as well:
\begin{corollary}\label{cor:main} The identity \eqref{eq: main_volume_eq}  holds even if $\xi$ is pseudoeffective.
\end{corollary}

Theorem~\ref{thm:main} and Corollary~\ref{cor:main} confirm \cite[Conjecture~1.4]{Den17} and also give an affirmative answer to the question of Lazarsfeld--Musta\c{t}\u{a} \cite[p. 831]{LM09}.

We note that, \emph{a priori} $X$ might not contain any proper complex submanifolds of positive dimension, let alone flags. For example, this happens if $X$ is a non-algebraic simple compact torus. However, one can always blow up $X$ at a point $p$. Then one can construct a flag $Y_\bullet$ on $\mathrm{Bl}_pX$ with $Y_1\cong\mathbb{CP}^{n-1}$ being the exceptional divisor and $Y_2\supset \dots\supset Y_n$ being a flag of linear subspaces in $Y_1$. Therefore, up to modifying $X$ and pulling back $\xi$, our main result is always applicable.

The proof of Theorem~\ref{thm:main} proceeds by induction on $\dim X$, with the recent work of the third named author \cite{WN21} helping with the induction step. In our proof, it will be necessary to approximate the volume of $\Delta(\xi)$ using the volume of partial Okounkov bodies, as studied in \cite{Xia21}. The advantage of our partial Okounkov bodies is that they have good bimeromorphic properties, allowing to reduce their study to the case when $\xi$ is K\"ahler. We will also study  currents with analytic singularities whose pushforward also have this property. This contrasts Deng's approach in the surface case \cite{Den17}, where he can apply the Zariski decomposition, without modifying the surface (cf. \cite{Bou04}).

In our approach, an extension result slightly generalizing \cite[Theorem 1.1]{CT14} will be necessary to relate the slice volume of the Okounkov body to the restricted volume of the big class (see Property~\hyperref[conj:restricted-vol-equality]{$C_n$}). Let us state this extension result,  which could be of independent interest, in light of  \cite[Conjecture 37]{DGZ16} and \cite{CGZ13}.

\begin{theorem}
\label{thm:CT-thm-refined} Let $V\subseteq X$ be a connected positive dimensional compact complex submanifold of $(X,\omega)$, and $T=\omega|_V+\ddc\varphi$ be a K\"ahler current on $V$. Assume that $\mathrm{e}^{\varphi}$ is a H\"older continuous function on $V$. Then one can find a K\"ahler current $\tilde T=\omega+\ddc\tilde \varphi$ on $X$ such that $\tilde \varphi|_V=\varphi$, i.e., $\tilde T$ extends $T$. Moreover, $\tilde\varphi$ is continuous on $X\setminus V$ and $\mathrm{e}^{\tilde \varphi}$ is H\"older continuous on $X$. In addition, if $\varphi$ has analytic singularities, then so does $\tilde\varphi$.
\end{theorem}

For the definition of analytic singularities we refer to Definition~\ref{def-loc-analy-sing}. Our definition is more general compared to \cite{CT14}; see  Remark~\ref{rem: smooth_remainder} for a comparison of various notions of analytic singularities appearing in the literature.

Potentials with exponential H\"older continuity often appear in the literature (see \cite{BFJ08, RW17}), so this result may find other applications in K\"ahler geometry.
Similar to \cite{CT14}, we prove our extension result using Richberg's gluing technique. We will adapt the exposition from \cite[Section I.5.E]{Dem12b}, which allows for some simplifications.

Note that one can not expect Theorem~\ref{thm:CT-thm-refined} to hold for big classes, due to the obvious obstruction coming from the non-K\"ahler locus. In fact one can hope to extend a K\"ahler current in a big class only if it has appropriate prescribed singularities along the non-K\"ahler locus. This technical point will arise in our proof of Theorem~\ref{thm:main}, and will be circumvented with the help of partial Okounkov bodies. \smallskip

In Section~\ref{sect_moment1}, we describe an alternative way of defining a convex body associated with a big class $\xi$. Instead of using the valuation vectors $\nu(T)$ of closed positive currents $T\in \xi$, we consider certain iterated toric degenerations of $X$, where in $n$ steps $(X,T)$ is degenerated to $(\PP^n,T_n)$. Here, $T$ is a closed positive $(1,1)$-current in $\xi$, while $T_n$ is a closed positive and \emph{toric} $(1,1)$-current on $\PP^n$. As described in Section~\ref{sect_moment1}, any such toric degeneration $(\PP^n,T_n)$ comes with a naturally associated moment body $\Delta(T_n)$. This allows us to define the moment body $\Delta^{\mu}(\xi)$ of $\xi$ as the closure of the union of the moment bodies of all such metrized toric degenerations of $(X,\xi)$. We note that an algebraic version of this construction is considered in the recent thesis of Murata \cite{Mur20}. We show that this alternative body agrees with the  Okounkov body of $\xi$:
\begin{theorem}\label{thm_momentbody} For $\xi$ big
the moment body $\Delta^\mu(\xi)$ coincides with the Okounkov body $\Delta(\xi)$.
\end{theorem}
As a direct consequence of Theorem~\ref{thm_momentbody} we get that $\Delta(T_n)\subseteq \Delta(\xi)$ for all toric degenerations $(\PP^n,T_n)$ of $(X,\xi)$, and that for any $\epsilon>0$ we can find a toric degeneration $(\PP^n,T_n)$ of $(X,\xi)$ such that the Hausdorff distance between $\Delta(T_n)$ and $\Delta(\xi)$ is less than $\epsilon$. This provides a strong link between transcendental Okounkov bodies and toric degenerations, which we think will be useful in further investigations.

The above result is inspired by the works of Harada--Kaveh \cite{HK15} and Kaveh \cite{Kav19}, who, building on earlier work of Anderson \cite{And13}, prove related results in the algebraic setting. This theorem is also related to work on canonical growth conditions \cite{WN18}.

\paragraph{Further directions} Finally, we mention several appealing directions that one could explore further. One may try to define the Chebyshev transform of \cite{WN14} in the transcendental setting to study weak geodesic rays in connection with transcendental notions of  K-stability attached to cscK metrics \cite{DR17,SD18}.
One could try to use Okounkov bodies to make progress on the conjectured transcendental Morse inequality \cite{BDPP13}, or to show differentiability of the volume in the big cone, akin to \cite[Corollary C]{LM09}. Moreover, one can study the mixed volume of Okounkov bodies and investigate its relation with the movable intersection product of big classes introduced in \cite{Bou02a}.

Regarding the extension of K\"ahler currents, it is intriguing to understand how far the Richberg type gluing techniques can be pushed. 

The results of this paper seem to naturally extend to manifolds of the Fujiki class~$\mathcal{C}$. Going beyond, we ask if one can attach Okounkov bodies to transcendental pseudoeffective cohomology classes $\xi$ on general compact complex manifolds.

\textbf{Acknowledgements.}We thank S\'ebastien Boucksom for suggesting Lemma~\ref{lem:push-forward-has-analy-sing}, simplifying our original argument, and for helpful discussions regarding analytic singularities. 
We also thank Zhiwei Wang for discussions on the extension of K\"ahler currents. We also thank the anonymous referees for suggestions on how to improve the presentation.

The first named author is partially supported by an Alfred P.
Sloan Fellowship and National Science Foundation grant DMS–2405274. The second
named author is supported by the Knut och Alice Wallenbergs Stiftelse. The third named author is partially supported by a grant from the Swedish Research Council and a grant from the Göran Gustafsson Foundation for Research in Natural Sciences and Medicine.
The fourth
named author is supported by Knut och Alice Wallenbergs Stiftelse grant KAW 2024.0273.
The fifth named author is supported by NSFC grants 12101052, 12271040, and 12271038.

\section{Preliminaries}

Let $(X,\omega)$ be an $n$-dimensional connected compact K\"ahler manifold. Let $\xi\coloneqq \{\theta\}\in H^{1,1}(X,\RR)$ be a big cohomology class. We write $\mathcal{Z}_+(X)$ for the set of closed positive $(1,1)$-currents on $X$ and $\mathcal{Z}_+(X,\xi)$ for the set of $T\in \mathcal{Z}_+(X)$ representing $\xi$. A function $\varphi:X\rightarrow [-\infty,\infty)$ is called quasi-plurisubharmonic (qpsh) if locally $\varphi=\phi+\psi$, where $\phi$ is smooth and $\psi$ is plurisubharmonic.
Put
\begin{align*}
\PSH(X,\theta)\coloneqq \{&\varphi:X\rightarrow [-\infty,\infty): \\
&\varphi \text{ is a qpsh function on $X$ such that }\theta+\ddc\varphi\geq 0 \},
\end{align*}
where $\ddc\coloneqq \sqrt{-1}\partial\bar\partial/2\pi$. Given a qpsh function $\varphi$, 
 the Lelong number of $\varphi$ at $x\in X$ is 
\[
\nu(\varphi,x)\coloneqq \sup\{c\geq0:\varphi(z)\leq c\log|z-x|^2+O(1) \text{ near }x\}.
\]
Given an irreducible analytic subset $E\subset X$, the generic Lelong number of $\varphi$ along $E$ is
\[
\nu(\varphi,E)\coloneqq \inf_{x\in E}\nu(\varphi,x).
\]
If $T=\theta+\ddc\varphi$ for $\varphi\in\PSH(X,\theta)$, then set $\nu(T,x)\coloneqq \nu(\varphi,x)$ and $\nu(T,E)\coloneqq \nu(\varphi,E)$.

If $\phi,\psi$ are qpsh, we say that $\phi$ is more singular than $\psi$ and write $\phi\preceq \psi$ if $\phi\leq \psi+C$ for some constant $C$. If $\phi\preceq \psi$ and $\psi\preceq \phi$, we then write $\phi\sim \psi$ and say that $\phi$ and $\psi$ have the same singularity type. This defines an equivalence relation.

Given $T,S\in \mathcal{Z}_+(X)$, with $T=\theta_1+\ddc\phi$ and $S=\theta_2+\ddc\psi$, $\theta_1,\theta_2$ being smooth, then we say that $T$ is more singular than $S$ if $\phi$ is more singular than $\psi$, and write $T\preceq S$. If $\phi\sim \psi$ we also say that $T$ and $S$ have the same singularity type and write $T\sim S$. 

\begin{lemma}\label{lma:Zariskimain}
    Let $\pi:Z\rightarrow X$ be a proper bimeromorphic morphism with $Z$ being a K\"ahler manifold. Then $\pi^*$ and $\pi_*$ induce a bijection between $\mathcal{Z}_+(X,\xi)$ and $\mathcal{Z}_+(Z,\pi^*\xi)$.
\end{lemma}
\begin{proof}
One can show that the pullback map $\pi^*$ is an isomorphism between $\PSH(X,\theta)$ and $\PSH(Z, \pi^*\theta)$, where $\theta$ is a smooth closed $(1,1)$-form representing $\xi$.

It is clear that the pullback map is injective. To show surjectivity, let $v \in \PSH(Z,\pi^* \theta)$. Let $U \subset X$ be the Zariski open set such that $\pi: \pi^{-1}(U) \to U$ is biholomorphic. Then $w \coloneqq v \circ \pi^{-1} \in \PSH(U,\theta)$, moreover $w$ is bounded from above. Since $X \setminus U$ is an analytic set, it follows that $w$ extends to an element of  $\PSH(X,\theta)$ \cite{Dem12b}. Since $\pi^*w = v$ away from an analytic set, it follows that $\pi^*w = v$ \cite{Dem12b}. 
\end{proof}

For a positive current $R\in \mathcal{Z}_+(X,\xi)$, let
$E_+(R)$ denote the locus where $R$ has positive Lelong number. The \emph{non-K\"ahler locus} $E_{nK}(\xi)$ and \emph{K\"ahler locus} $K(\xi)$ are  defined by
\[
E_{nK}(\xi)\coloneqq \bigcap_{\text{K\"ahler current }R\in \mathcal{Z}_+(X,\xi)}E_+(R), \quad
K(\xi)\coloneqq X\setminus E_{nK}(\xi).
\]
The class $\xi$ is K\"ahler if and only if $E_{nK}(\xi)=\varnothing$.

\begin{definition}\label{def-loc-analy-sing}
Let $q\in X$ be a point. A qpsh function $\varphi$ on $X$ is said to have \emph{(local) analytic singularities} near $q$ if there exists an open set $U$ containing $q$ such that
\[
\varphi=c\log\sum_i |f_i|^2+h,
\]
where $f_i\in\mathcal O_X(U)$ is a finite set of holomorphic functions, $c\in\QQ_{\geq 0}$, and $h$ is bounded.

If $\mathcal{I}$ is a coherent ideal sheaf on $X$ and $c\in \mathbb{Q}_{>0}$, we say that a qpsh function $\varphi$  on $X$ has \emph{analytic singularities (of type $(\mathcal{I},c)$)} if for any $q\in X$, we can locally write $\varphi$ as
\begin{equation*}
\varphi=c\log \sum_i |f_i|^2+h,
\end{equation*}
where $f_1,\ldots,f_N$ is a finite collection of local generators of $\mathcal{I}$, and $h$ is bounded.
\end{definition}

Similarly, we say a positive current $T\eqqcolon \theta+\ddc\varphi\in \mathcal{Z}_+(X,\xi)$ has analytic singularities near $q$ (resp. has analytic singularities (of type $(\mathcal{I},c)$)) if so does $\varphi$. We set
\[
\begin{aligned}
    \mathcal{A}(X,\xi,q)\coloneqq &  \left\{ T\in \mathcal{Z}_+(X,\xi): T \text{ has analytic singularities near $q$}\right\},\\
    \mathcal{A}(X,\xi)\coloneqq & \left\{ T\in \mathcal{Z}_+(X,\xi): T \text{ has analytic singularities}\right\},\\
    \mathcal{A}(X)\coloneqq &  \left\{ T\in \mathcal{Z}_+(X): T \text{ has analytic singularities}\right\}.
\end{aligned}
\]

\begin{remark}
\label{rmk:local-to-global-analy-sing-type} If $\varphi$ has local analytic singularities near every point of $X$, then there is a coherent ideal sheaf $\mathcal{I}$ and $c\in \mathbb{Q}_{>0}$ such that $\varphi$ has analytic singularities of type $(\mathcal{I},c)$.

    To see this, take a finite covering $\{U_j\}_j$ of $X$ such that $\varphi$ can be written as 
    \[
    c_j \log |F_j|^2+O(1),
    \]
    where $F_j=(f^j_{1},\ldots,f^j_{n_j})$ is a finite tuple of holomorphic functions on $U_j$. In the above representation we can always divide $c_j$ by an arbitrary $N\in \mathbb{Z}_{>0}$, by adjusting the $F_j$. Since $X$ is compact, we can guarantee that $c_j$ does not depend on $j$. We denote this common value by $c$.

    Now locally on $U_j$, the ideal sheaf $\mathcal{I}$ can be constructed as the integral closure of the ideal generated by $f^j_{1},\ldots,f^j_{n_j}$. As a consequence of Brian\c{c}on--Skoda theorem (see \cite[Corollary 11.18]{Dem15}), this gives a well-defined coherent ideal sheaf on $X$. But note that $T$ does not uniquely determine the pair $(\mathcal{I},c)$.
\end{remark}

Due to the above remark, when talking about currents $T$ with analytic singularities, we will often not specify the type $(\mathcal I,c)$.

We continue with a  well known observation, that follows from the short argument of \cite[Proposition~4.1.8]{Dem15}:

\begin{lemma}\label{lem: elem_log_ineq} If $u,v$ are qpsh functions with analytic singularities, then $\max(u,v)$ also has analytic singularities.
\end{lemma}

The following useful result due to Siu \cite{Siu74} will be also frequently used.

\begin{lemma}
\label{lem:Siu-deomp-lem}
    Assume $T\in\xi$ is a positive current and $D$ is a prime divisor on $X$. If $\nu(T,D)\geq a$ for some $a\in\QQ_{\geq 0}$, then $T-a[D]$ is a positive current in $\xi-a\{[D]\}$. Moreover, if $T$ has analytic singularities near a point, then so does $T-a[D]$.
\end{lemma}

\begin{proof}
    The first assertion follows from Siu's decomposition \cite{Siu74}.
    
Let $q \in X$. In a neighborhood of $q$ we have that
\[
u = c\log \sum_j |f_j|^2 + O(1).
\]

After rescaling we can assume that $a,c \in \mathbb Z$. Let $T = \theta+ \ddc u$ for some qpsh function $u$ and smooth closed $(1,1)$-form $\theta$.

By the Lelong--Poincar\'e formula there exists a neighborhood  $U$ of $q$ and $w \in \mathcal O(U)$ such that $\ddc \log |w|^2 = a [D]$.
By properties of Lelong numbers there exists $C>0$ such that
\[
u \leq \log |w|^2 + C.
\]
This implies that $|f_j|^{2c} \leq \tilde C|w|^2$ for all $j$, for some $\tilde C >0$. Due to the Riemann extension theorem, there exists $g_j \in \mathcal O(U)$ such that $f_j^c = w g_j$. As a result, the local singularities of $T - a [D]$ is modelled by $ \log \sum_j |g_j|^2$.
\end{proof}

We also need to refine the notion of analytic singularities, in a way that goes back to Demailly \cite{Dem92}. The definition below is a bit lengthy but turns out to be the most suitable one for this paper.

\begin{definition}
\label{def:analy-sing}
    Let $\mathcal I\subseteq\mathcal O_X$ be an analytic coherent ideal sheaf and $c\in\QQ_{>0}$. A qpsh function $\varphi$ on $X$ is said to have gentle analytic singularities of type $(\mathcal I,c)$ if 
\begin{enumerate}
    \item $\mathrm{e}^{{\varphi}/{c}}: X\to \RR_{\geq0}$ is a smooth function;
    \item locally $\varphi$ can be written as
    \begin{equation}
        \label{eq: anal_sing_I_def}
        \varphi=c\log\sum_i |f_i|^2+h,
    \end{equation}
where $\{f_i\}$ is a finite set of local generators of $\mathcal I$ and $h$ is a bounded remainder;
\item there is a proper bimeromorphic morphism $\pi: \tilde X\to X$ from a K\"ahler manifold $\tilde{X}$ and an effective $\mathbb{Z}$-divisor $D$ on $\tilde{X}$ such that one can write $\pi^*\varphi$ locally as
\[
\pi^*\varphi=c\log|g|^2+h,
\]
where $g$ is a local equation of the divisor $D$ and $h$ is smooth.
\end{enumerate}
\end{definition}

A positive current $T=\theta+\ddc\varphi \in \mathcal{Z}_+(X,\xi)$ is said to have \emph{gentle analytic singularities of type $(\mathcal{I},c)$} if so does $\varphi$. We set
\[
\Omega(X,\xi)\coloneqq \left\{\text{K\"ahler current in $\xi$ with gentle analytic singularities}\right\}.
\]

\begin{remark}\label{rem: smooth_remainder} Let $\varphi$ be a qpsh function with analytic singularities of type $(\mathcal I, c)$ such that the remainder $h$ is smooth in \eqref{eq: anal_sing_I_def}. Then $\varphi$ is easily seen to have gentle analytic singularities. The reverse direction is not true however, as the smoothness of the remainder depends on the choice of generators for $\mathcal I$ (we are grateful to S. Boucksom for helpful discussions regarding this point). 
For instance one can look at the potentials $\log(|z_1|^2+|z_2|^2)=\log(|z_1|^2+|2z_2|^2)+\log\frac{|z_1|^2+|z_2|^2}{|z_1|^2+|2z_2|^2}$ in $\CC^2$; the first has 0 as its smooth remainder while the remainder $\log\frac{|z_1|^2+|z_2|^2}{|z_1|^2+|2z_2|^2}$ of the second potential is not smooth.
This motivates the use of gentle analytic singularities throughout the paper.
\end{remark}

In the whole paper, we are using the word \emph{modification} in a 'restrictive' sense: We shall say that $\pi \colon Y\rightarrow X$ is a modification if it is bimeromorphic, and it is a composition of a sequence of blowing-ups with smooth centers. Observe that by the construction in \cite[Theorem~2.0.2]{Wlo09}, the embedded resolutions where the ambient space is compact are always modifications in this sense. Note that if $X$ is a K\"ahler manifold, then so is $Y$.

It follows from Definition~\ref{def:analy-sing} that one can always resolve the singularities of K\"ahler currents with gentle analytic singularities. More precisely, by \cite[Theorem~3.6.1]{Wlo09}, for any $R\in\Omega(X,\xi)$, one can find a modification $\pi:Y\to X$  so that
\[
\pi^*R=\beta+[D],
\]
where $\beta$ is a smooth semipositive $(1,1)$-form on $\tilde X$ and $D$ is an effective $\QQ$-divisor on $Y$.

A classical regularization result due to Demailly \cite{Dem92} shows that there are plenty of K\"ahler currents in $\xi$ with gentle analytic singularities. We refer the reader to \cite[Theorem 3.2]{DP04} for a proof of the following version of Demailly's regularization theorem, and also \cite[Corollary 13.13]{Dem15} for a detailed exposition. 

\begin{theorem}
\label{thm:Dem-reg}
	Let $T=\theta+\ddc\varphi\in \mathcal{Z}_+(X,\xi)$ be a K\"ahler current. Then one can construct a sequence $T_k=\theta+\ddc\varphi_k$, which are K\"ahler currents with gentle analytic singularities of type $(\mathcal I(k\varphi),\frac{1}{k})$ and $\varphi\leq\varphi_k$ for all $k\gg0$. Here $\mathcal I(k\varphi)$ denotes the multiplier ideal sheaf of $k\varphi$ that is generated by holomorphic functions $f\in\mathcal O_X$ such that $|f|^2\mathrm{e}^{-k\varphi}\in L^1_{\text{loc}}$.
\end{theorem}

As a direct consequence, we have the following result.

\begin{lemma}\label{lem:sing-decrease}
	For any two K\"ahler currents $T_1,T_2\in\mathcal{Z}_+(X,\xi)$, one can find $T_3\in\Omega(X,\xi)$ that is less singular than both $T_1$ and $T_2$.
\end{lemma}

\begin{proof}
	Write $T_i=\theta+\ddc\varphi_i$ for $i=1,2$. Put $\varphi_0\coloneqq \max\{\varphi_1,\varphi_2\}$. Then $T\coloneqq \omega+\ddc\varphi_0$ is a K\"ahler current that is less singular than both $T_1$ and $T_2$. Note that $T$ may not have gentle analytic singularities, but then we can apply Theorem~\ref{thm:Dem-reg} to find another K\"ahler current $T'\in\Omega(X,\xi)$ that is less singular than $T$.
\end{proof}

This lemma together with the Noetherian property of analytic sets imply that one can find $R\in\Omega(X,\xi)$ such that $E_+(R)=E_{nK}(\xi)$. So in particular, $E_{nK}(\xi)$ is a closed analytic subset of $X$ (see \cite[Theorem~3.17]{Bou04}).

\begin{remark}\label{rmk:pullbackanaly}
By definition, if $T\in \mathcal{A}(X)$ has type $(\mathcal{I},c)$ and $g:Y\rightarrow X$ is a proper bimeromorphic morphism and $Y$ is a K\"ahler manifold, then $g^*T$ has type $(\mathcal{O}_Y\cdot g^{-1}\mathcal{I},c)$.

In fact, if we locally write $T$ as $\ddc \varphi$ for some $\varphi$ satisfying $\varphi=c\log (|f_1|^2+\cdots+|f_N|^2)+O(1)$, then $g^*T$ can locally be written as $c\log (|g^*f_1|^2+\cdots+|g^*f_N|^2)+O(1)$. The ideal generated by $g^*f_1,\ldots,g^*f_N$ is just $\mathcal{O}_Y\cdot g^{-1}\mathcal{I}$.
\end{remark}

By the above remark, the pullback of currents with analytic singularities is well-behaved. The same can not be said about pushforwards. However, the following result suggested by Boucksom shows that in certain particular situations good behaviour can be expected:
 \begin{lemma}
 \label{lem:push-forward-has-analy-sing}
    Assume that $\pi:Z\to X$ is a proper bimeromorphic morphism and $Z$ is a K\"ahler manifold.  Let $\xi$ be a big class on $X$ and $D$ an effective $\mathbb{Q}$-divisor on $Z$. Assume that
       $\pi^*\xi-\{D\}$ contains a smooth K\"ahler form $\beta$. Let $T\coloneqq \beta+[D]$.
        Then $\pi_* T$ belongs to $\mathcal{A}(X,\xi)$.
\end{lemma}

\begin{proof}
By rescaling, we may assume that $D$ is an integral divisor.

After a further rescaling of $D$, we claim that
$-D$ is $\pi$-globally generated in the following sense: For all $z\in Z$ there exists an open neighborhood $U \subset X$ of $\pi(z)$ such that the sections $\mathcal O (-D)(\pi^{-1}(U))$ generate $\mathcal O (-D)_z$. Let us assume the claim for the moment. It implies that: 
\begin{equation}\label{eq: sheaf_id} \mathcal{O}_Z\cdot\pi^{-1}\pi_*\mathcal{O}_Z(-D) = \mathcal O_Z(-D).
\end{equation}

By definition, $T$ has analytic singularities of type $\mathcal{O}_Z(-D)$. We want to show that $\pi_*T\in \alpha$ has analytic singularities of type $\pi_*\mathcal{O}_Z(-D)$.

By \cite[Proposition~2.2]{Bou02b},
there is a closed positive current $S$ on $X$ with analytic singularities of type $\pi_*\mathcal{O}_Z(-D)$. Then $\pi^*S$ has analytic singularities of type $\mathcal{O}_Z\cdot\pi^{-1}\pi_*\mathcal{O}_Z(-D)$ by Remark~\ref{rmk:pullbackanaly}.

From \eqref{eq: sheaf_id} we obtain that $\pi^*S\sim T= \pi^*\pi_*T$ (the latter equality follows from Lemma~\ref{lma:Zariskimain}). It follows that $S\sim \pi_*T$. So $\pi_*T$ has analytic singularities of type $\pi_*\mathcal{O}_Z(-D)$.

Finally, we address the claim. This is a simple consequence of \cite[Proposition~1.4]{Nak87} when $\pi$ is projective. Here we give a different proof by adapting the $L^2$ techniques of \cite[Theorem~6.27]{Dem15}. 

Let $\omega$ be a K\"ahler metric of $X$,  satisfying $\theta < \omega$, where $\theta$ is a smooth closed $(1,1)$-form in the class of $\xi$.

Let $U \subset X$ be a coordinate ball neighborhood of $\pi(z)$. We can assume existence of $v \in C^\infty(U)$ such that $\omega = \ddc v$. 

Let $h$ be a hermitian metric of $\mathcal O(-D)$ on $Z$ such that $\mathrm{i}\Theta(h) = \beta - \pi^*\theta$. We consider $\tilde h = h \mathrm{e}^{-v \circ \pi}$, a hermitian metric on $\mathcal O(-D)|_{\pi^{-1}(U)}$. We get that $\tilde h$ is a positive metric, since $\mathrm{i}\Theta(h \mathrm{e}^{-v \circ \pi}) = \beta + \pi^*(\omega - \theta)\geq \beta > 0$.

Let $f_z\geq 0$ be a  smooth function on $X$ such that $ f_z(x) = |x - z|^{2n +1}$ in a neighborhood of $z$, and $f_z$ is bounded away from zero, away from $z$. Since $Z$ is compact, we can construct $f_z$ in such a manner that there exists $\varepsilon >0$ and $C> 1$, both independent of $z$, such that:
\begin{equation}\label{eq: C_2_est}
f_z|_{B(z,\varepsilon)} = |x - z|^{2n +1}, \ \ \|f_z\|_{C^2(Z)} \leq C, \ \ \sup_{X \setminus B(z,\varepsilon)} \bigg|\frac{1}{f_z}\bigg| \leq C,
\end{equation}
where $B(z,\varepsilon)$ is a geodesic ball centered at $z$, with radius $\varepsilon>0$.

 Let $m$ be an arbitrary hermitian metric on $K_Z^{-1}$. Let $\hat h^k \coloneqq  \tilde h^k \otimes m  \mathrm{e}^{-\log f_z}$ be a hermitian metric on $\mathcal O(-kD) \otimes K_Z^{-1}|_{\pi^{-1}(U)}$. We have that 
 \[
 \mathrm{i}\Theta(\tilde h^k) = \mathrm{i}\Theta(h^k) + \mathrm{i}\Theta(m) + \ddc \log f_z \geq k \beta + \mathrm{i}\Theta(m) + \ddc \log f_z.
 \]
 Since $\log f_z$ is psh in $B(z,\varepsilon)$, for $k$ big enough (only dependent on the constant $C$ of \eqref{eq: C_2_est}) $\mathrm{i}\Theta(\hat{h}^k)$ is a strongly positive metric.
 
As a result of the compactness of $Z$, taking a sufficiently divisible $k$, we can replace $D$ (resp. $\theta$, $\omega$, $\beta$) with $kD$ (resp. $k\theta$, $k\omega$, $k\beta$) and assume that $\hat h$ is a strongly positive (singular) metric of $\mathcal O(-D) \otimes K_Z^{-1}$ on ${\pi^{-1}(U)}$.

Let $q \in \mathcal O(-D)_z = \mathcal (O(-D) \otimes K_Z^{-1} \otimes K_Z)_z$. Let $\phi:Z\rightarrow \mathbb{R}_{\geq 0}$ be a  smooth function supported in a neighbourhood of $z$, identically equal to $1$ near $z$, so that $\phi q \in C^\infty(\pi^{-1}(U), O(-D))$.
Then $\bar \partial (\phi q)$ is a smooth $1$-form on $\pi^{-1}(U)$, with values in  $\mathcal O(-D) \otimes K_Z^{-1} \otimes K_Z \otimes \mathcal I(\mathrm{e}^{-\log f_z})$.

Since $U$ is a coordinate ball, we get that $\pi^{-1}(U)$ is weakly pseudoconvex, so \cite[Corollary 5.3]{Dem15} is applicable to yield existence of a smooth section $g$ of $\mathcal O(-D) \otimes K_Z^{-1} \otimes K_Z \otimes \mathcal I(\mathrm{e}^{-\log f_z})$, defined on $\pi^{-1}(U)$ with $\overline{\partial}g=\bar\partial(\phi q)$ on $\pi^{-1}(U)$.

Due to our choice of $f_z$,   $g$ must vanish at $z$. Hence, $w \coloneqq \phi q - g \in H^0(\pi^{-1}(U), O(-D))$, moreover $( q-w)_z \in \mathfrak{m}_{Z,z} \cdot \mathcal O(-D)_z$. 

Since $\mathcal O(-D)_z$ is a finite $\mathcal O_{Z,z}$-module, we apply 
Nakayama's lemma \cite[Proposition~2.8]{AM69} to conclude that $\mathcal O (-D)(\pi^{-1}(U))$ generates $\mathcal O (-D)_z$, proving the claim.
\end{proof}

Finally, we recall the notion of non-pluripolar volumes.

\begin{definition}[\cite{Bou02b}]
	The volume $\vol(\xi)$ of pseudoeffective class $\xi$ is zero if $\xi$ is not big. If $\xi$ is big then
\[
\vol(\xi)\coloneqq \sup_{T\in\Omega(X,\xi)}\int_{X}T^n.
\]
\end{definition}
Here and in the sequel, products of the form $T^n$ are interpreted as non-pluripolar complex Monge--Amp\`ere measures, as introduced in \cite{BEGZ10}. 
It is shown in \cite{Bou02b} that the volume function $\vol(\cdot)$ thus defined is continuous in $H^{1,1}(X,\RR)$.

\begin{definition}[\cite{His12,Mat13,WN21,CT22}]
	Assume $Z\subset X$ is a closed complex submanifold of dimension $d$. If $Z\cap K(\xi)\neq\varnothing$,
then the restricted volume $\vol_{X|Z}(\xi)$ of $\xi$ along $Z$ is 
\[
\vol_{X|Z}(\xi)\coloneqq \sup_{T\in\Omega(X,\xi)}\int_{Z}T|_Z^d.
\]
If $Z\cap K(\xi+\varepsilon\{\omega\})\neq\varnothing$ for any $\varepsilon>0$, then set
\[
\vol_{X|Z}(\xi)\coloneqq \lim_{\varepsilon\to 0}\vol_{X|Z}(\xi+\varepsilon\{\omega\}).
\]
Otherwise, set $\vol_{X|Z}(\xi)\coloneqq 0$.
\end{definition}
Note that the restricted volume $\vol_{X|Z}(\xi)$ could be strictly smaller than the volume of the restricted class $\xi|_Z$.
In general, the restricted volume is not easy to compute. But when $\xi$ is semipositive, one simply has $\vol_{X|Z}(\xi)=\vol(\xi|_Z)=\int_Z\xi^d$. Also recall that, in the open cone of big classes whose non-K\"ahler locus does not contain $Z$, the function $\vol_{X|Z}(\cdot)$ is continuous \cite[Corollary 4.11]{Mat13}.

The following result was proved by the third named author \cite{WN21} and extended very recently by Vu \cite{Vu23} (see also \cite[Theorem~1.5]{Den17} for the surface case and \cite{BFJ09,LM09,WN19a} for the case of projective manifolds). It will play a crucial role in proving Theorem~\ref{thm:main}.

\begin{theorem}
\label{thm:WN-volume-identity}
	Assume that $Z$ is a connected smooth hypersurface in $X$ such that $Z\cap K(\xi)\neq\varnothing$. Put $\nu_{\max}(\xi,Z)\coloneqq \sup\{\nu(T,Z):T\in\Omega(X,\xi)\}$.
 Then one has
	\begin{equation*}
	\frac{\mathrm d}{\mathrm dt}\bigg|_{t=0}\vol(\xi-t\{Z\})=-n\vol_{X|Z}(\xi) \ \ \ 
    \end{equation*}
    and 
    \begin{equation*}
   \vol(\xi)=n\int_0^{\nu_{\max}(\xi,Z)}\vol_{X|Z}(\xi-t\{Z\})\,\mathrm dt.
    \end{equation*}
\end{theorem}
\begin{remark}
    According to the recent result of \cite{Vu23}, the assumption on $Z$ can be relaxed: it suffices to assume that $Z$ is an irreducible hypersurface in $X$. In particular, the case where $Z$ is contained in the non-K\"ahler locus of $\xi$ is also allowed.
\end{remark}

\begin{proof}
    The first point is due to \cite[Theorem C]{WN21}. The second point is also contained in \cite{WN21}. We recall its proof for the reader's convenience. 
    
    We claim that $Z\cap K(\xi-t\{Z\})\neq\varnothing$ for any $t\in[0,\nu_{\max}(\xi,Z))$. Indeed, for any small $\varepsilon>0$, one can find $T\in\Omega(X,\xi)$ such that $\nu(T,Z)=\nu_{\max}(\xi,Z)-\varepsilon$. And also, due to $Z\cap K(\xi)\neq\varnothing$, there exists $S\in\Omega(X,\xi)$ with $\nu(S,Z)=0$. Considering linear interpolation between $S$ and $T$ shows that for any $t\in(0,\nu_{\max}(\xi,Z))$ one can find a K\"ahler current $G\in\mathcal{Z}_+(X,\xi)$ so that $\nu(G,Z)=t$. Then Siu's decomposition implies that $G-t[Z]$ is a K\"ahler current with zero Lelong number along $Z$, showing that $Z\cap K(\xi-t\{Z\})\neq\varnothing$.
    
    On the other hand we have $\vol(\xi-t\{Z\})\searrow0$ as $t\nearrow\nu_{\max}(\xi,Z)$. If this were not the case, then $\xi-\nu_{\max}(\xi,Z)\{Z\}$ would have to be a big class. For small $\varepsilon>0$, $\xi-(\nu_{\max}+\varepsilon)(\xi,Z)\{Z\}$ must also be big, and contains a K\"ahler current $Q$ with gentle analytic singularities. Choosing $\varepsilon$ so that $\nu_{\max}+\varepsilon\in\QQ$, then $Q+(\nu_{\max}+\varepsilon)[Z]$ is an element in $\Omega(X,\xi)$ with Lelong number larger than $\nu_{\max}(\xi,Z)$ along $Z$, a contradiction.

    Now for any $c\in(0,\nu_{\max}(\xi,Z))$, we deduce from the above that
    \[
    \vol(\xi)-\vol(\xi-c\{Z\})=n\int_0^c\vol_{X|Z}(\xi-t\{Z\})\,\mathrm dt,
    \]
    where the integrand is continuous by \cite[Corollary 4.11]{Mat13}. Sending $c\nearrow \nu_{\max}(\xi,Z)$ we conclude the second point.
\end{proof}

\section{Transcendental Okounkov bodies}
\label{sec:construct-body}

\paragraph{The big case.} Let  $\xi$ a big $(1,1)$-class on $X$. The goal of this paragraph is to attach natural convex bodies (compact convex sets with non-empty interior) $\Delta(\xi)\subset\RR^n$ to $\xi$, which will be called the \emph{Okounkov bodies} of $\xi$. One starts with a flag 
\[
Y_\bullet=(X=Y_0\supset Y_1\supset \dots\supset Y_{n-1}\supset Y_n=\{q\}),
\] 
where $Y_i$ is a smooth connected complex submanifold of $X$ with codimension $i$. We call such a flag a \emph{smooth flag} on $X$ centered at $q$.

However, it is possible that $X$ might not contain any proper submanifolds of positive dimension, let alone flags. As we have pointed out in the introduction, an easy remedy for this issue is to replace $X$ by a smooth model $X'$ and then one can construct a flag $Y_\bullet$ on $X'$. So after possibly modifying $X$ and pulling back $\xi$, in what follows we always assume that $X$ contains a flag $Y_\bullet$, which is centered at the point $q\in Y_n$.

Next, using $Y_\bullet$ we shall associate an Okounkov body $\Delta(\xi)=\Delta_{Y_\bullet}(\xi)\subset\RR^n$ with $\xi$. The construction is essentially the same as in \cite{Den17}, apart from some minor modifications. 

Following \cite{LM09,Den17}, we define a valuative map $\nu: \mathcal{A}(X,\xi,q)\to \QQ_{\geq 0}^n$ by
\[
T\mapsto \nu(T)=\nu_{Y_\bullet}(T)=(\nu_1(T),\dots,\nu_n(T))
\]
as follows. First, set
\[
\nu_1(T)\coloneqq \nu(T,Y_1).
\]
Since the generic Lelong number of $T_1\coloneqq (T-\nu_1(T)[Y_1])$ along $Y_1$ is zero and this current has analytic singularities near $q$, its potential is not identically $-\infty$ on $Y_1$, hence   $T_1\coloneqq (T-\nu_1(T)[Y_1])|_{Y_1}$ is a positive current on $Y_1$, with analytic singularities near $q$.

Next put
\[
\nu_2(T)\coloneqq \nu(T_1,Y_2),
\]
and continue in this manner to define all the remaining values $\nu_i(T)\in\QQ_{\geq 0}$. 

Note that $\Omega(X,\xi)\subseteq\mathcal{A}(X,\xi,q)$. So the map $\nu$ is also defined on $\Omega(X,\xi)$.
Consider the image sets
\[
\nu(\Omega(X,\xi))\text{ and }\nu(\mathcal{A}(X,\xi,q)).
\]
These two are bounded sets in $\QQ^n_{\geq 0}$, since all the Lelong numbers appearing in the above construction are bounded by cohomological constants (see e.g. \cite[Lemma 4.3]{Den17}).
Moreover, they are convex sets in $\QQ_{\geq 0}^n$, since $\nu(tT_1+(1-t)T_2)=t\nu(T_1)+(1-t)\nu(T_2)$ for $T_1,T_2$ in $\Omega(X,\xi)$ (or in $\mathcal{A}(X,\xi,q)$) and $t\in[0,1]\cap\QQ$.

The following lemma plays a key role in this paper.
\begin{lemma}
\label{lem:omega-dense-in-gamma}
	For any $T\in\mathcal{A}(X,\xi,q)$, one can find a sequence $T_k\in\Omega(X,\xi)$ such that $\nu(T_k)\to\nu(T)$ in $\QQ_{\geq 0}^n$.
\end{lemma}

\begin{proof}

First, we might as well assume that $T$ is a K\"ahler current. Indeed, one can take $S\in\Omega(X,\xi)$ and let $T_\varepsilon\coloneqq (1-\varepsilon)T+\varepsilon S$ for $\varepsilon\in(0,1)\cap\QQ$. So $T_\varepsilon$ is a K\"ahler current in $\mathcal{A}(X,\xi,q)$ and $\nu(T_\varepsilon)=(1-\varepsilon)\nu(T)+\varepsilon\nu(S)$ can be made arbitrarily close to $\nu(T)$.

	Next assume that $T=\theta+\ddc\varphi$ with $\varphi=c\log\sum_{i=1}^r|g_i|^2+O(1)$ near $q$, where $g_1,\dots,g_r\in\mathcal O_{X,q}$ and $c\in\QQ_{\geq 0}$. Let $\cJ\subseteq\mathcal O_{X,q}$ be the ideal generated by $g_1,\dots,g_r$. Up to rescaling the class $\xi$ we will assume that $c=1$ for simplicity. Let $T_k=\theta+\ddc\varphi_k$ be the sequence given by Theorem~\ref{thm:Dem-reg}. Near $q$, one can write $\varphi_k=\frac{1}{k}\log\sum_{l}|g_{kl}|^2+O(1)$ for some local generators $\{g_{kl}\}_l$ of $\mathcal I(k\varphi)_q$. Let $m\coloneqq \min\{n-1,r-1\}$. Then by the Brian\c{c}on--Skoda theorem (see \cite[Theorem 11.17]{Dem15}), the equality of ideals
	\[
 \cI ((k+m)\varphi)_q=\cJ^k\cdot\cI(m\varphi)_q
 \]
	 holds at $q$ for any large $k\in\NN$. Let $\{h_j\}$ be a set of generators of $\cI(m\varphi)_q$ and put $\psi\coloneqq \log\sum_j|h_j|^2$, which is a qpsh function defined near $q$. Then the above equality implies that
	 \[
  \varphi_{k+m}=\frac{k}{k+m}\varphi+\frac{1}{k+m}\psi+O(1)
  \]
	  holds near $q$. So sending $k\to\infty$ we find that $\nu(T_k)\to\nu(T)$, as desired.
\end{proof}

Consequently, we are led to the following definition (cf. \cite[Definition 4.12]{Den17}).
\begin{definition}
\label{def:Ok-body}
	The Okounkov body $\Delta(\xi)=\Delta_{Y_\bullet}(\xi)$ of $\xi$ is defined to be 
 \[
 \Delta(\xi)\coloneqq \overline{\nu(\Omega(X,\xi))}=\overline{\nu(\mathcal{A}(X,\xi))}=\overline{\nu(\mathcal{A}(X,\xi,q))}.
 \]
 Here the closure is taken with respect to the Euclidean topology of $\RR^n$.
\end{definition}

As in \cite[Proposition 4.13]{Den17}, we have the following properties for $\Delta(\xi)$:

\begin{proposition}
\label{prop:body-properties}
	The following properties hold:
	\begin{enumerate}
	    \item For $\lambda\in\RR_{>0}$, one has $\Delta(\lambda\xi)=\lambda\Delta(\xi)$;
	    \item $\Delta(\xi_1)+\Delta(\xi_2)\subseteq\Delta(\xi_1+\xi_2)$ for two big classes $\xi_1$ and $\xi_2$;
		\item $\Delta(\xi)=\bigcap_{\varepsilon>0}\Delta(\xi+\varepsilon\{\omega\})$ for any K\"ahler form $\omega$ on $X$.
	\end{enumerate}
\end{proposition}

\begin{proof}
	Point (1) is clear when $\lambda\in\QQ_{>0}$. If $\lambda\in\RR_{>0}$, it suffices to observe that, for any $T\in\Omega(\lambda\xi)$, one can find a sequence $T_i\in\Omega(X,\xi)$ so that $\lambda\nu(T_i)\to\nu(T)$. Point (2) is also clear. So it remains to argue (3).

	First, it is easy to see that $\Delta(\xi)\subseteq\Delta(\xi+\varepsilon\{\omega\})$. On the other hand, $(1+\sqrt{\varepsilon})\xi-
	(\xi+\varepsilon\{\omega\})$ is a big class for small $\varepsilon>0$. So from (1) and (2) we have that $(1+\sqrt{\varepsilon})^n\vol_{\RR^n}(\Delta(\xi))\geq\vol_{\RR^n}(\Delta(\xi+\varepsilon\{\omega\}))$. Thus, we get that 
 \[
 \vol_{\RR^n}(\Delta(\xi))=\lim_{\varepsilon\to 0}\vol_{\RR^n}(\Delta(\xi+\varepsilon\{\omega\})).
 \]
 So $\Delta(\xi)=\bigcap_{\varepsilon>0}\Delta(\xi+\varepsilon\{\omega\})$ follows (here we used that $\Delta(\xi)$ is a convex body, by Proposition~\ref{prop:delta-is-convex-body} below).
\end{proof}

\begin{proposition}
\label{prop:delta-is-convex-body}
	When $\xi$ is big, $\Delta(\xi)$ is a convex body in $\RR^n$, so $\vol_{\RR^n}(\Delta(\xi))>0$.
\end{proposition} 

This result is also from \cite{Den17}. Below we give a slightly  different proof.

\begin{proof}
We need to show that $\Delta(\xi)$ has non-empty interior.
Since $\xi$ can be written as the sum of a K\"ahler class and a big class, it suffices to show that $\Delta(\xi)$ has non-empty interior when $\xi$ is K\"ahler. 	
	
	So assume that $\xi$ is K\"ahler. Fix a K\"ahler form $\omega\in\xi$. We will argue that $\Delta(\xi)$ contains a small simplex near the origin. Clearly, $\Delta(\xi)$ contains the origin. Applying \cite[Lemma~2.1(i)]{DP04} to $Y_i$, we can find qpsh functions $\psi_i$ on $X$ for $1\leq i\leq n$ so that $\psi_i$ has analytic singularities precisely along $Y_i$. Choosing $\varepsilon>0$ small enough so that $T_i\coloneqq \omega+\varepsilon\ddc\psi_i$ are all elements in $\Omega(X,\xi)$. Then one has $\nu(T_i)=\varepsilon e_i\in\RR^n_{\geq 0}$, so that $\Delta(\xi)$ contains a simplex.	
\end{proof}

The following result could be of independent interest.

\begin{proposition}
	The map $\xi\mapsto\Delta(\xi)$ from the big cone of $X$ to $\RR^n$ is continuous with respect to the Hausdorff topology in $\RR^n$.
\end{proposition}

\begin{proof}
 Choose smooth $(1,1)$-forms $\eta_1,\dots,\eta_r$ on $X$ so that the $\{\eta_i\}$ form a basis of $H^{1,1}(X,\RR)$. For any two big classes $\xi, \xi'\in H^{1,1}(X,\RR)$ write $\xi=(a_1,\dots,a_r)$ and $\xi'=(a_1',\dots,a_r')$ using this basis and let $\|\xi-\xi'\|\coloneqq \max_i|a_i-a_i'|$.
We wish to argue that the Hausdorff distance between $\Delta(\xi)$ and $\Delta(\xi')$ tends to zero as $\|\xi-\xi'\|\to 0$. 

Fix a K\"ahler current $T\in\xi$ and assume that $T\geq\omega$ for some K\"ahler form $\omega$ on $X$. Pick $C>1$ so that $-C\omega\leq\eta_i\leq C\omega$ holds for all $1\leq i\leq r$. Then $T'\coloneqq T+\sum_i(a_i'-a_i)\eta_i$ is a K\"ahler current in $\xi'$ as long as $\|\xi-\xi'\|<(Cr)^{-1}$. Moreover, we see that $(1+\|\xi-\xi'\|^{1/2})\xi-\xi'$ is big as well when $\|\xi-\xi'\|<(Cr)^{-2}$, since $(1+\|\xi-\xi'\|^{1/2})T-T'$ is a K\"ahler current in that class. This implies that $\Delta(\xi')\subseteq(1+\|\xi-\xi'\|^{1/2})\Delta(\xi)$ as long as $\|\xi-\xi'\|<(Cr)^{-2}$. On the other hand, observe that  $(1+\|\xi-\xi'\|^{1/2})T'-T$ is a K\"ahler current when $\|\xi-\xi'\|<(2Cr)^{-2}$, so that $\Delta(\xi)\subseteq(1+\|\xi-\xi'\|^{1/2})\Delta(\xi')$ when $\|\xi-\xi'\|<(2Cr)^{-2}$. 

The above discussion easily implies that the Hausdorff distance between $\Delta(\xi)$ and $\Delta(\xi')$ tends to zero as $\|\xi-\xi'\|\to0$. 
\end{proof}

\paragraph{The pseudoeffective case.} In case $\xi$ is only pseudoeffective and not big, there are different ways to associate a compact convex set with $\xi$ \cite{CHPW18}. In \cite[Definition 4.23]{Den17} Deng suggests the definition below that we will follow in this work. For a K\"ahler form $\omega$, let
\begin{equation}\label{eq: psef_Okounkov}\Delta(\xi) \coloneqq  \bigcap_{\varepsilon>0} \Delta(\xi + \varepsilon \{\omega\}).
\end{equation}
Due to Proposition~\ref{prop:body-properties}, the above definition does not depend on the choice of $\omega$. Another natural choice is 
\[
\tilde \Delta(\xi) \coloneqq \overline{\nu(\mathcal A(X,\xi))}.
\]
Since $\mathcal A(X,\xi) \subset \mathcal A(X,\xi + \varepsilon \{\omega\})$ we obtain that $\tilde \Delta(\xi) \subset \Delta(\xi)$. 
Notice that $\tilde \Delta(\xi)$ might be empty.

In case $\xi$ is pseudoeffective and non-big, we will establish later that $\vol_{\mathbb R^n} \Delta(\xi) =0 = \vol(\xi)$, so this will imply that $\vol_{\mathbb R^n} \tilde \Delta(\xi) =0 = \vol(\xi)$.

In particular, Theorem~\ref{thm:main} will remain valid, regardless what definition one chooses for the convex compact sets associated to pseudoeffective (non-big) classes.

\paragraph{The conjectured volume identity.}  Let us return to a big class $\xi$. When $\xi$ is in the N\'eron--Severi space it was shown in  \cite{Den17} that $\vol(\xi)=n!\vol_{\RR^{n}}(\Delta(\xi))$. For a transcendental big class $\xi$, following \cite[Conjecture~1.4]{Den17}, we expect the same:

\begin{namedtheorem}[$A_n$]
\label{conj:vol-equality}
Assume that $\dim X=n$ and $\xi$ is big. Then $\vol(\xi)=n!\vol_{\RR^{n}}(\Delta(\xi))$.
\end{namedtheorem}
Note that the ensemble of Property~\hyperref[conj:vol-equality]{$A_n$}'s for all $n\geq 1$ is equivalent to Theorem~\ref{thm:main}.

It is easy to establish this when $n=1$. We recall its proof for the reader's convenience.

\begin{lemma}\label{lem: A_1}Property~\hyperref[conj:vol-equality]{$A_{1}$} holds.
\end{lemma}

\begin{proof}
When $\dim X =1$, $H^{1,1}(X,\RR)\cong H^2(X,\RR)\cong\RR$. A class $\xi\in H^{1,1}(X,\RR)$ is K\"ahler if and only if $\int_X\xi>0$. Moreover, in this setting every big class $\xi =\{\theta\}$ is a K\"ahler class, and a flag is simply a choice of point $Y_1 = X \supset Y_0 = \{q\}$.

For any $t\in(0,\int_X\theta)$, $\{\theta\}-t\{q\}$ the class is K\"ahler, which then contains a smooth K\"ahler form $\omega_t$. Then $\omega_t+t[q]\in\Omega(X,\{\theta\})$ and it has Lelong number $t$ at $q$. So $[0,\int_X\theta]\subseteq\Delta(\{\theta\})$. On the other hand, one has $\nu(T,q)\leq\int_X\theta$ for any $T\in\Omega(X,\{\theta\})$, otherwise $\{\theta\}-\int_X\theta\{q\}$ is a K\"ahler class, which is absurd. So we have $\Delta(\{\theta\})=[0,\int_X\theta]$.
\end{proof}

 The following sections aim to prove Property~\hyperref[conj:vol-equality]{$A_n$} for higher dimensions. It is notable that, by using the divisorial Zariski decomposition for surfaces, it was shown in \cite[Theorem 1.8]{Den17} that Property~\hyperref[conj:vol-equality]{$A_2$} holds as well.

\section{Lifting the flags}
In this section, we fix a smooth flag $Y_{\bullet}=(Y_0\supset \cdots \supset Y_n)$ on $X$.

If $\pi:Y\rightarrow X$ is a proper bimeromorphic morphism of complex manifolds and $Z$ is an analytic set in $X$ not contained in the exceptional locus $E$, then we define the strict transform of $Z$ in $Y$ as the closure of $\pi^{-1}Z\setminus \pi^{-1}E$ in $Y$. This is an analytic set in $Y$ by \cite[Corollary~5.4]{Dem12b}. When $\pi$ is a modification, this coincides with the notion defined in \cite[Section~3, strict transforms]{BM97} by \cite[Remark~3.15]{BM97}.

\begin{definition}
    Let $\pi:Z\rightarrow X$ be a proper bimeromorphic morphism with $Z$ being a compact K\"ahler manifold. We say that a smooth flag $W_{\bullet}=(W_0\supset \cdots \supset W_n)$ is a \emph{lifting} of $Y_{\bullet}$ to $Z$ if  the restriction of $\pi$ to $W_i\rightarrow Y_i$ is defined and bimeromorphic. 
\end{definition}

\begin{lemma}\label{lma:flagvaluationchange}
    Let $\pi:Z\rightarrow X$ be a proper bimeromorphic morphism with $Z$ being a K\"ahler manifold. Assume that there is a lifting $W_{\bullet}=(W_0\supset \cdots \supset W_n)$ of $Y_{\bullet}$ to $Z$. Then there is a matrix $g\in \mathrm{SL}_n(\mathbb{Z})$ of the form $I+N$ with $N$ being strictly upper triangular (namely, all entries on and below the diagonal are $0$) such that
    for any $T\in \mathcal{A}(X)$, we have
    \[
    \nu_{W_{\bullet}}(\pi^*T)= \nu_{Y_{\bullet}}(T)g.
    \]
\end{lemma}
Here the valuations of currents are considered as row vectors. For simplicity, we will also say $(W_{\bullet},g)$ is a \emph{lifting} of $Y_{\bullet}$ to $Z$.

Observe that matrices of the form $I+N$ are exactly the automorphisms of the ordered Abelian group $\mathbb{Z}^n$ with respect to the lexicographic order. So this result can be seen as a generalization of \cite[Theorem~2.9]{CFKL17}.
\begin{proof} We take induction on $n=\dim X$. The case $n=1$ is trivial. In general, assume that the result is proved when $\dim X = n-1$. 

For simplicity, we write $\nu=\nu_{Y_{\bullet}}$ and $\nu'=\nu_{W_{\bullet}}$.
Let $\mu$ (resp. $\mu'$) be the valuation of currents defined by $Y_1\supset \cdots \supset Y_n$ (resp. $W_1\supset \cdots \supset W_n$). Then we need to show that
    \[
    \begin{bmatrix}
     \nu'(\pi^*T)_1 &
     \mu'((\pi^*T-\nu'(\pi^*T)_1[W_1])|_{W_1})
    \end{bmatrix}
    \]
    equals
    \[
    \begin{bmatrix}
        \nu(T)_1 &
        \mu((T-\nu(T)_1[Y_1]))|_{Y_1}
    \end{bmatrix}g.
    \]
    Clearly
    \[
    \nu'(\pi^*T)_1=\nu(T)_1 \eqqcolon c.
    \]
By the inductive hypothesis, we can find $h\in \mathrm{SL}_{n-1}(\mathbb{Z})$ of the form $I$ plus a strictly upper triangular matrix such that
\begin{equation}\label{eq: ind_hypos}
        \mu'(\Pi^{*}(T-c[Y_1])|_{Y_1})=\mu((T-c[Y_1])|_{Y_1})h,
\end{equation}
    where $\Pi:W_1\rightarrow Y_1$ is the restriction of $\pi$. Observe that
    \begin{align*}
        \Pi^{*}(T-c[Y_1])|_{Y_1}&= (\pi^*(T-c[Y_1]))|_{W_1}\\ &=(\pi^*T-c[W_1])|_{W_1}+c(\pi^*[Y_1]-[W_1])|_{W_1}.
    \end{align*}
    So
    \[
    \mu'(\Pi^{*}(T-c[Y_1])|_{Y_1})=\mu'((\pi^*T-c[W_1])|_{W_1})+c\mu'((\pi^*[Y_1]-[W_1])|_{W_1}).
    \]
Combining the above with \eqref{eq: ind_hypos}, we see that the following choice of $g$ satisfies the requirements of the lemma:    \[
    g\coloneqq 
    \begin{bmatrix}
    1 & -\mu'((\pi^*[Y_1]-[W_1])|_{W_1})\\
    0 & h
    \end{bmatrix}.
    \]
\end{proof}

\begin{theorem}\label{thm:liftflaggeneral}
Assume that $\pi:Z\rightarrow X$ is a modification of $X$, then there is a modification $\pi':Z'\rightarrow Z$ such that the flag $Y_{\bullet}$ admits a lifting $W_{\bullet}$ to $Z'$. 
\end{theorem}
By Hironaka's Chow lemma \cite[Corollary~2]{Hir75}, we can relax the assumption to that $\pi$ is a proper bimeromorphic morphism from a reduced complex space $Z$. 

When $X$ is projective, this result is already known, see \cite[Theorem~2.9]{CFKL17} for example.

The proof below relies on the  embedded resolutions studied in \cite{BM97} and \cite{Wlo09}. Let $Z$ be an analytic subset of a compact complex manifold $M_1$, that in turn is a submanifold of a compact complex submanifold $M_2$. Due to \cite[Theorem~2.0.2(4)]{Wlo09} there exist canonical desingularizations $\sigma_1: M'_1 \to M_1$ and $\sigma_2: M'_2 \to M_2$ of $Z$ inside $M_1$ and $M_2$ respectively, so that a canonical embedding $M'_1 \hookrightarrow M'_2$ exists and  makes the following diagram commute:
\begin{equation}\label{eq: Wlod}
\begin{tikzcd}
M_1' \arrow[d,"\sigma_1"] \arrow[r, hook] & M_2' \arrow[d,"\sigma_2"] \\
M_1 \arrow[r, hook]            & M_2. 
\end{tikzcd}
\end{equation}
Moreover, $M'_1 \hookrightarrow M'_2$ maps the (desingularized) strict transforms of $Z$ onto each other.
\begin{proof}
We begin by setting $W_0=Z$. We will construct $W_i$ inductively for each $i$. Assume that for $0\leq i<n$ a smooth partial flag $W_0\supset\cdots \supset W_i$ has been constructed on a modification $\pi_i:Z_i\rightarrow Z$ so that $\pi\circ \pi_i$ restricts to bimeromorphic morphisms $W_j\rightarrow Y_j$ for each $j=0,\ldots,i$. 

By Zariski's main theorem, $W_i\rightarrow Y_i$ is an isomorphism outside a codimension $2$ subset of $Y_i$.
We let $W_{i+1}$ be the strict transform of $Y_{i+1}$ in $W_i$.
The problem is that $W_{i+1}$ is not necessarily smooth. 

We will further modify $Z_i$ and lift $W_1,\ldots, W_{i+1}$ in order to make the flag smooth. Take the  embedded resolution of $(W_{j},W_{i+1})$, say $W_j'\rightarrow W_j$ for each $j=0,\ldots,i$.

Due to \eqref{eq: Wlod}, we have  canonical embeddings 
$W_{i}'\hookrightarrow W_{i-1}'\hookrightarrow\cdots \hookrightarrow W_0'$
making the following diagram commutative:
\[
\begin{tikzcd}
W_i' \arrow[d] \arrow[r, hook] & W_{i-1}' \arrow[d] \arrow[r, hook] & \cdots \arrow[r, hook] \arrow[d, "\vdots", phantom] & W_0' \arrow[d] \\
W_i \arrow[r, hook]            & W_{i-1} \arrow[r, hook]            & \cdots \arrow[r, hook]                              & W_0           
\end{tikzcd}
\]
Let $W_{i+1}'$ be the strict transform of $W_{i+1}$ in  $W_i'$. It suffices to define $\pi_{i+1}$ as the morphism $W_0'\rightarrow Z_i\rightarrow Z$ and 
replace $W_0\supset \cdots\supset W_{i+1}$ by $W_0'\supset \cdots\supset W_{i+1}'$. 
\end{proof}

As a consequence, we have the following useful approximation result.

\begin{proposition}\label{prop:approx-push-forward-current} Let $\pi:Z\rightarrow X$ be a modification of $X$ and $\xi$ be a big class on $X$. Assume that $Y_{\bullet}$ admits a lifting $(W_{\bullet},g)$ to $Z$. For any  $T\in \Omega(Z,\pi^*\xi)$ we can find  K\"ahler currents $T_k\in \mathcal{A}(X,\xi)$ such that
    \begin{enumerate}
    \item $\lim_{k\to\infty}\nu_{Y_{\bullet}}(T_k)=\nu_{W_{\bullet}}(T) g^{-1}$;
    \item 
    $T_k\preceq \pi_*T$.    
    \end{enumerate} 
\end{proposition}
In particular, it follows that
\[
\Delta_{W_{\bullet}}(Z,\pi^*\xi)=\Delta_{Y_{\bullet}}(X,\xi)g.
\]
This is the transcendental generalization of the birational invariance of the classical Okounkov bodies.
\begin{proof}
    By Theorem~\ref{thm:liftflaggeneral} there exists a modification $\pi':Z'\rightarrow Z$ such that 
    \begin{enumerate}[(i)]
        \item $\pi'^*T=\beta+[D]$,
    where $\beta$ is a smooth semipositive $(1,1)$-form on $Z'$ and $D$ is an effective $\mathbb{Q}$-divisor;
    \item $Y_{\bullet}$ admits a lifting $(V_{\bullet},h)$ to $Z'$.
  \end{enumerate}

    Observe that $(V_{\bullet},g^{-1}h)$ is a lifting of $Y_{\bullet}$ to $Z'$.

    Note that $\pi'$ is a composition of blowups with smooth centers.
     By the proof of \cite[Theorem 3.5]{DP04}, one can find an effective exceptional divisor $E$ on $Z'$ and a smooth $(1,1)$-form $\eta$ in $\{E\}$ such that
     $\pi'^*T-\varepsilon\eta$ is a K\"ahler current for any small $\varepsilon>0$. This implies that $\beta-\varepsilon\eta$ is a smooth K\"ahler form. Let us put for $k\gg1$
     \[
     \beta_k\coloneqq \beta-k^{-1}\eta\text{ and }D_k\coloneqq D+k^{-1}E.
     \]
Then $\beta_k+[D_k]\in(\pi\circ\pi')^*\xi$ and by Lemma~\ref{lem:push-forward-has-analy-sing} one has $T_k\coloneqq (\pi\circ\pi')_*(\beta_k+[D_k])\in\mathcal{A}(X,\xi)$.

     We compute
     \begin{equation*}
         \begin{aligned}
             \nu_{Y_{\bullet}}(T_k)&=\nu_{V_{\bullet}}(\beta_k+k^{-1}[E]+[D])h\\
             &=\nu_{V_{\bullet}}(\pi'^*T)h+k^{-1}\nu_{V_{\bullet}}([E])h=\nu_{W_{\bullet}}(T)g+k^{-1}\nu_{V_{\bullet}}([E])h.
         \end{aligned}
     \end{equation*}
     Letting $k\to\infty$ we conclude the first point. The second point is clear from the construction of $T_k$.
\end{proof}

\section{Partial Okounkov bodies}
\label{sec:rel-Ok-body}

For this section let $\xi$ be a big class on $X$. Fix a smooth flag $Y_{\bullet}=(Y_0\supset \cdots \supset Y_n )$. We will show that  Property~\hyperref[conj:vol-equality]{$A_n$} holds with the help of partial Okounkov bodies, introduced and studied recently in \cite{Xia21} in case of big line bundles. 

\begin{definition}
\label{def:rel-Ok-body}
For $R\in\Omega(X,\xi)$, we set 
\[
\Omega(X,\xi;R)\coloneqq \left\{ T\in \Omega(X,\xi): T\preceq R\right\}.
\]
The \emph{partial Okounkov body} of $R$ is defined as the following closed convex subset of $\RR^n_{\geq 0}$:
\[
\Delta(\xi;R)=\Delta_{Y_\bullet}(\xi;R)\coloneqq \overline{\nu_{Y_\bullet}(\Omega(X,\xi;R))}.
\]
\end{definition}
We will fix $R\in\Omega(X,\xi)$ throughout this section.

\begin{proposition}
	The partial Okounkov body $\Delta(\xi;R)$ is a convex body in $\RR_{\geq 0}^n$.
\end{proposition}

\begin{proof}
	Since $R$ is a K\"ahler current, we may fix a small $\varepsilon>0$ such that $R-\varepsilon\omega$ is still a K\"ahler current. This implies that $\Delta(\xi;R)$ contains the set $\nu(R)+\varepsilon\Delta(\{\omega\})$. As we have seen in the proof of Proposition~\ref{prop:delta-is-convex-body}, the body $\Delta(\{\omega\})$ contains a simplex, hence so does $\Delta(\xi;R)$.
\end{proof}

\begin{example}\label{ex:Okoun_divplusbeta}
When $R=[D]+\beta$, where $D$ is a $\mathbb{Q}$-divisor on $X$ and $\beta\in \mathcal{Z}_+(X)$ is a current with bounded potentials, we have
\[
\Delta_{Y_\bullet}(\xi;R)=\Delta_{Y_\bullet}(\xi-\{D\})+\nu_{{Y_\bullet}}([D]).
\]
Indeed, suppose that $T\in \Omega(X,\xi-\{D\})$, then $T+[D]\in \Omega(X,\xi;R)$, so 
\[
\Delta_{Y_\bullet}(\xi;R)\supset\Delta_{Y_\bullet}(\xi-\{D\})+\nu_{{Y_\bullet}}([D]).
\]
On the other hand, if $T\in \Omega(X,\xi;R)$, then $T-[D]\in \mathcal{A}(X,\xi-\{D\})$, due to Lemma~\ref{lem:Siu-deomp-lem}. It follows from Lemma~\ref{lem:omega-dense-in-gamma} that $\nu(T)\in \Delta_{Y_\bullet}(\xi-\{D\})+\nu_{{Y_\bullet}}([D])$.
\end{example}

Let us put
\[
\mathcal{A}(X,\xi;R)\coloneqq \left\{T\in \mathcal{A}(X,\xi):T\preceq R \right\}
\]
It is helpful to know the following equivalent formulation of $\Delta(\xi;R)$.

\begin{lemma}\label{lma:pob_loc_analytic}
One has
\[
\Delta(\xi;R)=\overline{\nu(\mathcal{A}(X,\xi;R))}.
\]
\end{lemma}

\begin{proof}
The proof uses similar strategy as in Lemma~\ref{lem:omega-dense-in-gamma}.
 It suffices to argue that $\nu(G)\in\Delta(\xi;R)$ for any $G\in\mathcal{A}(X,\xi;R)$.
We might as well assume that $G$ is a K\"ahler current, otherwise one can replace $G$ by $(1-\varepsilon)G+\varepsilon R$. 

Write $R=\theta+\ddc u$ and $G=\theta+\ddc \varphi$, with $\varphi\leq u$. Since they are both K\"ahler currents, we can assume that $R\geq\omega$ and $G\geq\omega$. For any sufficiently small rational number $\varepsilon>0$, put $u_\varepsilon\coloneqq (1+\varepsilon)u$ and $\varphi_\varepsilon\coloneqq \varphi+\varepsilon u$. Then one has
	$R_\varepsilon\coloneqq \theta+\ddc u_\varepsilon\geq\frac{1}{2}\omega$ and $G_\varepsilon\coloneqq \theta+\ddc \varphi_\varepsilon\geq\frac{1}{2}\omega$. Let $G_{\varepsilon,k}\coloneqq \theta+\ddc\varphi_{\varepsilon,k} $ be the regularization sequence of $G_\varepsilon$ given in Theorem~\ref{thm:Dem-reg}, which are K\"ahler currents with analytic singularities of type $(\mathcal I(k\varphi_\varepsilon),\frac{1}{k})$. Since $\mathcal I(k\varphi_\varepsilon)\subseteq\mathcal I(ku_\varepsilon)$ and $u_\varepsilon$ has analytic singularities, we can apply Brian\c{c}on--Skoda theorem as in the proof of Lemma~\ref{lem:omega-dense-in-gamma} to see that
	$
	\varphi_{\varepsilon,k+m}\leq\frac{k}{k+m}u_\varepsilon+O(1)
	$
	holds on $X$, where $m\in\NN$ is independent of $k$ (but it may depend on $\varepsilon$). Then for all sufficiently large $k$, one has $\varphi_{\varepsilon,k}\leq u+O(1)$, i.e., $G_{\varepsilon,k}\in\Omega(\xi,R)$. Moreover, by the proof of Lemma~\ref{lem:omega-dense-in-gamma}, $\nu_{Y_\bullet}(G_\varepsilon)=\lim_{k\to\infty}\nu_{Y_\bullet}(G_{\varepsilon,k})$, since $G_\varepsilon$ has analytic singularities; meanwhile, it is clear that $\nu_{Y_\bullet}(G)=\lim_{\varepsilon\to 0}\nu_{Y_\bullet}(G_\varepsilon)$. So the desired assertion follows.
\end{proof}

\begin{proposition}
\label{prop:rel-body-exhaust-body}
	One can find an increasing sequence $R_i\in\Omega(X,\xi)$ such that
	\begin{enumerate}
		\item $E_+(R_i)=E_{nK}(\xi)$;
		\item $\int_{X}R^n_i\nearrow\vol(\xi)$;
		\item $\Delta(\xi)=\overline{\bigcup_i\Delta(\xi;R_i)}$, so that $\vol_{\RR^n}(\Delta(\xi;R_i))\nearrow\vol_{\RR^n}(\Delta(\xi))$.
	\end{enumerate}
\end{proposition}

\begin{proof}
	Since $\nu(\Omega(X,\xi))\subseteq\QQ_{\geq 0}^n$ is a countable set, we can write $\nu(\Omega(X,\xi))=\{a_1,a_2,\dots\}$. Pick $T_i\in\Omega(X,\xi)$ such that $\nu(T_i)=a_i$ and also pick a sequence $G_i\in\Omega(X,\xi)$ such that $\int_{X}G_i^n\to\vol(\xi)$. Choose $R_0\in\Omega(X,\xi)$ such that $E_+(R_0)=E_{nK}(\xi)$. Then applying Lemma~\ref{lem:sing-decrease} repeatedly, we choose $R_i\in\Omega(X,\xi)$ so that it is less singular than $T_i,G_i$ and $R_{i-1}$. So $R_i$ is an increasing sequence, which clearly satisfies (1) and (3). That (2) holds follows from the monotonicity of Monge--Amp\`ere masses, see \cite[Theorem~1.16]{BEGZ10} or \cite[Theorem~1.2]{WN19b}. 
\end{proof}

\begin{lemma}\label{lma:bim_inv_pob}
    Let $R\in\Omega(X,\xi)$. Take a modification $\pi:Z\rightarrow X$ such that the flag $Y_{\bullet}$ admits a lifting $(W_{\bullet},g)$ to $Z$ and 
    $\pi^*R=\beta+[D]$, where $\beta$ is a smooth semipositive form on $Z$ and $D$ is an effective $\QQ$-divisor. Then
    \begin{equation}\label{eq:DeltaxiR}
    \Delta_{Y_\bullet}(\xi;R)g=\Delta_{W_\bullet}(\{\beta\})+\nu_{W_\bullet}([D]).
    \end{equation}
\end{lemma}

\begin{proof}
    
    For any $T\in\Omega(X,\xi;R)$, by Lemma~\ref{lem:Siu-deomp-lem} $\pi^*T-[D]$ is a positive current with analytic singularities in $\{\beta\}$. So one has
    \[
    \nu_{W_\bullet}(\pi^*T)=\nu_{W_\bullet}(\pi^*T-[D])+\nu_{W_\bullet}([D])\in\Delta_{W_\bullet}(\{\beta\})+\nu_{W_\bullet}([D]).
    \]
    Therefore, $\nu_{Y_\bullet}(T)\in \big(\Delta_{W_\bullet}(\{\beta\})+\nu_{W_\bullet}([D])\big)g^{-1}$,
    implying that
    \[
    \Delta_{Y_\bullet}(\xi;R)\subseteq \big(\Delta_{W_\bullet}(\{\beta\})+\nu_{W_\bullet}([D])\big)g^{-1}.
    \]

On the other hand, for any K\"ahler current $S\in\Omega(Z,\{\beta\})$, note that $S+[D]$ belongs to $\Omega(Z,\pi^*\xi)$ and is more singular than $\pi^*R$. 

By Proposition~\ref{prop:approx-push-forward-current}, there exist  K\"ahler currents $T_k\in \mathcal{A}(X,\xi)$ s.t. 
$ \lim_{k\to\infty} \nu_{Y_{\bullet}}(T_k)=\nu_{V_{\bullet}}(S+[D])g^{-1}$ and $T_k\preceq R$.

It follows from Lemma~\ref{lma:pob_loc_analytic} that
$\nu_{W_{\bullet}}(S+[D])\in \Delta_{Y_\bullet}(\xi;R)g$,
finishing the proof of \eqref{eq:DeltaxiR}:
\[
\big(\Delta_{W_\bullet}(\{\beta\})+\nu_{W_\bullet}([D])\big)g^{-1}\subseteq \Delta_{Y_\bullet}(\xi;R).
\]
\end{proof}

We hope to show that partial Okounkov bodies have the following property:

\begin{namedtheorem}[$B_n$]
\label{conj:relative-Oko-vol-equal}
    Assume that $\dim X=n$ and $\xi$ is big. For any $R\in\Omega(X,\xi)$ one has
\begin{equation*}
     \int_{X}R^n=n!\vol_{\RR^n}(\Delta(\xi;R)).
\end{equation*}
\end{namedtheorem}

\begin{lemma}
\label{lma:rel-vol-conj=vol-conj}
    Property~\hyperref[conj:vol-equality]{$A_n$} and Property~\hyperref[conj:relative-Oko-vol-equal]{$B_n$} are equivalent.
\end{lemma}

\begin{proof}
    That Property~\hyperref[conj:relative-Oko-vol-equal]{$B_n$} implies Property~\hyperref[conj:vol-equality]{$A_n$} is immediate thanks to Proposition~\ref{prop:rel-body-exhaust-body}. Let us prove the other direction.

    Fix $R\in\Omega(X,\xi)$. By Theorem~\ref{thm:liftflaggeneral}, there is a modification $\pi:Z\rightarrow X$ such that 
    \begin{enumerate}[(i)]
        \item $\pi^*R=\beta+[D]$, where $\beta$ is a smooth semipositive form on $Z$ and $D$ is an effective $\QQ$-divisor;
        \item The flag $Y_{\bullet}$ admits a lifting $(W_{\bullet},g)$ to $Z$.
    \end{enumerate}
    By Lemma~\ref{lma:bim_inv_pob}, $\Delta_{Y_\bullet}(\xi;R)g=\big(\Delta_{W_\bullet}(\{\beta\})+\nu_{W_\bullet}([D])\big)$. Since $\det g =1$, we obtain that:
    \begin{align*}
    \vol_{\RR^n}(\Delta_{Y_\bullet}(\xi;R))=\vol_{\RR^n}(\Delta_{Y_\bullet}(\xi;R)) \cdot \det g&=\vol_{\RR^n}(\Delta_{W_\bullet}(\{\beta\}))\\ &=\frac{1}{n!}\int_Z\beta^n\\ &=\frac{1}{n!}\int_{X}R^n.
    \end{align*}
\end{proof}

We shall see in \S \ref{sec:vol-eq} that  Property~\hyperref[conj:relative-Oko-vol-equal]{$B_n$} holds for all $n\geq1$, and so does ~\hyperref[conj:vol-equality]{$A_n$}.

\section{Extension of K\"ahler currents}
In order to prove Property~\hyperref[conj:vol-equality]{$A_n$}, the natural idea is to argue by induction on $n$. Then we need to investigate the K\"ahler currents on submanifolds and see if they can be extended to K\"ahler currents on the ambient space.
This motivates us to prove the following extension theorem, generalizing \cite[Theorem 1.1]{CT14}.

\begin{theorem}
\label{thm:CT-thm-refined'}(=Theorem~\ref{thm:CT-thm-refined}) Let $V\subseteq X$ be a connected positive dimensional compact complex submanifold of $(X,\omega)$. Let $T=\omega|_V+\ddc\varphi$ be a K\"ahler current on $V$. Assume that $\mathrm{e}^{\varphi}$ is a H\"older continuous function on $V$. Then one can find a K\"ahler current $\tilde T=\omega+\ddc\tilde \varphi$ on $X$ such that $\tilde \varphi|_V=\varphi$, i.e., $\tilde T$ extends $T$. Moreover, $\tilde\varphi$ is continuous on $X\setminus V$ and $\mathrm{e}^{\tilde \varphi}$ is H\"older continuous on $X$. In addition, if $\varphi$ has analytic singularities, then so does $\tilde\varphi$.
\end{theorem}

\begin{proof}
Thanks to \cite[Lemma 2.1]{DP04}, we can find a qpsh potential $F$ on $X$, such that $F|_V\equiv -\infty$, $F$ is smooth on $X\setminus V$ and $\omega+\ddc F$ is a K\"ahler current with analytic singularities precisely along $V$. Moreover, $\mathrm{e}^{F}: X\to \RR_{\geq0}$ is a H\"older continuous function.

Choose $\varepsilon>0$ small enough so that
\[
\omega|_V+\ddc\varphi\geq3\varepsilon\omega|_V\text{ and }
\omega+\ddc F\geq 3\varepsilon\omega.
\]

We also choose finitely many open coordinate balls $\{W_j\}_{1\leq j\leq r}$ covering $V$, such that on each $W_j$ there are local coordinates $(z_1,\dots,z_n)$ so that $V\cap W_j=\{z_1=\dots=z_{n-k}=0\}$, where $k=\dim V$. Write $z=(z_1,\dots,z_{n-k})$ and $z'=(z_{n-k+1},\dots,z_n)$.

Assume that $\mathrm{e}^{\varphi}$ is $\alpha$-H\"older continuous for some $\alpha>0$. Due to \cite[Lemma 2.1(i)]{DP04} we can rescale $F$ to arrange the existence of $c \in (0,\alpha)\cap \mathbb{Q}$ such that on $W_j$ we have
\begin{equation}\label{eq: F_local}
F(z) = \frac{c}{2} \log \sum_{i=1}^{n-k} |z_i|^2 + O(1).
\end{equation}

We define a function $\varphi_j$ on $W_j$ by
\begin{equation}
    \label{eq:def-phi-j}
    \varphi_j(z,z')\coloneqq \log(\mathrm{e}^{\varphi(z')}+\mathrm{e}^{F(z,z')})+A\cdot\mathrm{dist}((z,z'),V)^2,
\end{equation}
where  $\mathrm{dist}$ is the Riemannian distance. After shrinking $W_j$'s
 slightly, still preserving $V\subset\cup_jW_j$, we can fix $A>1$ sufficiently large so that $\varphi(z') + A\cdot\mathrm{dist}((z,z'),V)^2$ and $F(z,z') + A\cdot\mathrm{dist}((z,z'),V)^2$ are $(1 - 2\varepsilon)\omega$-psh (see \cite[\S 2]{NWZ24}). Then
 \[
 \omega+\ddc\varphi_j\geq 2\varepsilon\omega
 \]
 holds on $W_j$ for all $j$, due to \cite[Theorem I.4.16]{Dem12b}.

From \eqref{eq:def-phi-j} one sees that $\varphi_j$ is continuous on $W_j\setminus V$ and $\varphi_j|_V=\varphi$. If in addition $\varphi$ has analytic singularities, then due to Lemma~\ref{lem: elem_log_ineq},
the potential $\varphi_j$ has local analytic singularities on $W_j$ (in the sense of Definition~\ref{def-loc-analy-sing}), as $\varphi_j$ has the same singularity type as $\max(\varphi(z'),F(z,z'))$. 

Next we need to fix slightly smaller open  coordinate balls $W_j'\Subset U_j \Subset T_j \Subset W_j$ such that $\cup_j W'_j$ still covers $V$.

Let $f_j \in C^\infty(W_j)$, $f_j \geq 0$, compactly supported such that $\supp f_j \subset {W'_j}^c$ and $f_j > 1$ in a neighborhood of $\overline  {T_j \setminus U_j}$.

Let $\eta >0$ small, so that $\ddc \eta f_j \leq  \varepsilon \omega$. We introduce
\[
\tilde \varphi_j \coloneqq  \varphi_j - \eta f_j \in \PSH(W_j, (1-\varepsilon)\omega).
\]

Let  $W \coloneqq \cup_j T_j$ and for $x \in W$ we introduce the function
\begin{equation}\label{eq: psi_ext_def}
\psi(x) \coloneqq  \max_{ x \in T_l} \tilde \varphi_l(x).
\end{equation}
We will show that $\psi \in \PSH(G,(1-\varepsilon)\omega)$ for some neighborhood $G \subset W$ of $V$,  moreover $\psi|_V = \varphi$ and $\psi$ is continuous on $G\setminus V$.

For any $x \in W'_j$ let $I_x \subset \{1,\ldots,r \}$ be the collection of indices $l$ such that
$x \in \overline{T_l \setminus U_l}$.

Fixing $x \in W'_j$, there exists an open coordinate ball $B_x \subset W'_j$, centered at $x$  such that  for all $y \in B_x$ we have $I_y \subset I_x$.

After shrinking $B_x$ further, we can assume that $ \varphi_l -\eta \geq \tilde \varphi_l$ on $B_x$, for any $l \in I_x$. Indeed, this follows from the fact that $f_l > 1$ in a neighborhood of $\overline  {T_l \setminus U_l}$, containing $x$.

In case $x \in V \cap W'_j$, due to Lemma~\ref{lem:phi>phi-eta}, we can shrink $B_x\subset W'_j$ further so that
\begin{equation}\label{eq: open_ineq}
 \tilde \varphi_j(z) =\varphi_j(z)> \varphi_l(z) - \eta > \tilde \varphi_l(z), \ \ z \in B_x, \ l \in I_x.
\end{equation}

We denote by $J_x \subset \{1,\ldots, r\} \setminus I_x$  the collection of indices $l$ such that $T_l \cap B_x$ is non-empty. We claim that $B_x \subset T_l$ for any $l \in J_x$. If this were not the case then there would exist $y \in B_x$ such that $y \in \partial T_l \subset \overline{T_l \setminus U_l}$. Hence, $l \in I_y \subset I_x$, a contradiction. As a result, $J_x$ is also the subset of indices $l$ in $\{1,\ldots, r\} \setminus I_x$ such that $B_x \subset T_l$. This will be used crucially in \eqref{eq: psi_qpsh} below.

Let $B_{x_1}, \ldots, B_{x_p}$ be a covering of $V$ with $x_i \in W'_{j_i} \cap V$. We take 
\[
G \coloneqq  \bigcup_j B_{x_j}.
\]

Clearly, $\psi \geq \tilde \varphi_j = \varphi_j$ on $W_j'$. As $\varphi_l - \eta f_l \leq \varphi$ on $V \cap T_j$ for all $l$, we get $\psi = \varphi$ on $V$.

Revisiting \eqref{eq: psi_ext_def}, due to \eqref{eq: open_ineq}, on each $B_{x_i}$ we have

\begin{equation}\label{eq: psi_qpsh}
\psi|_{B_{x_i}} \coloneqq   \max \left\{\tilde \varphi_{j_i}|_{B_{x_i}}, \max_{k \in J_{x_i}}\tilde \varphi_k|_{B_{x_i}}\right\} = \max \left\{\varphi_{j_i}|_{B_{x_i}}, \max_{k \in J_{x_i}}\tilde \varphi_k|_{B_{x_i}}\right\}.
\end{equation}
Since the right-hand side above is clearly a $(1-\varepsilon)\omega$-psh function on $B_{x_i}$, we get that $\psi \in \PSH(G,(1-\varepsilon)\omega)$. Moreover, $\psi|_V=\varphi$ and $\psi$ is continuous on $G\setminus V$.

After slightly shrinking $G$, we can assume that $\psi$ is defined and continuous in a neighborhood of $\partial G$.
Let $k>0$ and introduce $\tilde \varphi \coloneqq  \max(\psi, F + k)$ on $G$ and $\tilde \varphi \coloneqq  F +k$ on $X \setminus G$. For $k$ big enough $\tilde \varphi \in \PSH(X,(1-\varepsilon)\omega)$. 

If in addition $\varphi$ has analytic singularities, then thanks to Lemma 
\ref{lem: elem_log_ineq} so does $\tilde\varphi$, since due to \eqref{eq: psi_qpsh} the potential $\tilde \varphi$ is locally the maximum of finitely many potentials with analytic singularities. Note also that the maximum operation preserves exponential H\"older continuity, so all the potentials involved in the proof are exponentially H\"older continuous. Therefore, the potential $\tilde\varphi$ satisfies all the required properties.
\end{proof}

As promised in the above argument, we deliver the following key lemma:

\begin{lemma}
\label{lem:phi>phi-eta} With the notation of  the above proof, the following property holds: for any $1\leq i,j\leq r$, any point $p\in W_i\cap W_j\cap V$ and $\eta>0$, there is an open neighborhood of $p$ on which $\varphi_j\geq\varphi_i-\eta$ holds.
\end{lemma}

\begin{proof} Let $p_i : W_i \to V$ and $p_j : W_j \to V$ be the coordinate projections onto  $V$ in each respective chart.
To show that  $\varphi_j\geq\varphi_i-\eta$ in a neighborhood of $p$, it is enough to show that in a neighborhood of $p$ we have
\[
\mathrm{e}^{ \varphi(p_j(x))}  + \mathrm{e}^{F(x)} \geq \mathrm{e}^{-\eta}\big(\mathrm{e}^{ \varphi(p_i(x))} + \mathrm{e}^{F(x)}\big).
\]
This in turn is equivalent to 
\begin{equation}\label{eq: extKey}
\big(1- \mathrm{e}^{-\eta} \big)\mathrm{e}^{F(x)} \geq \mathrm{e}^{- \eta}\mathrm{e}^{ \varphi(p_i(x))} - \mathrm{e}^{ \varphi(p_j(x))}.
\end{equation}
Since $\mathrm{e}^{-\eta}< 1$, we claim that in fact the following slightly stronger inequality holds, in a neighborhood of $p$:
\[
\big(1- \mathrm{e}^{-\eta} \big)\mathrm{e}^{F(x)} \geq \mathrm{e}^{- \eta}\big( \mathrm{e}^{ \varphi(p_i(x))} - \mathrm{e}^{ \varphi(p_j(x))}\big).
\]

Due to \eqref{eq: F_local} there exists  $D>0$ such that 
\[
\big(1- \mathrm{e}^{-\eta} \big)\mathrm{e}^{F(x)}\geq D\cdot\textup{dist}(x,V)^c
\]
for all $x \in W_i \cap W_j \setminus V$.
As a result, it is enough to argue that  in a neighborhood of $p$ we have: 
\begin{equation}
    \label{eq:dist-Holder-ineq}
    D\cdot\textup{dist}(x,V)^c \geq \mathrm{e}^{- \eta}\big( \mathrm{e}^{ \varphi(p_j(x))} - \mathrm{e}^{ \varphi(p_i(x))}\big).
\end{equation}
This follows from the $\alpha$-H\"older continuity of the expression on the right-hand side and its vanishing on $V$ (recall that $0< c < \alpha$). So the proof is complete.
\end{proof}

\begin{remark}
    The key inequality \eqref{eq: extKey}
does not hold when $\mathrm{e}^\varphi$ is not continuous. Here is a counterexample: let $V=\CC$, $X=\CC^2$. Let $\varphi=\sum_{n=1}^\infty\frac{1}{n^2}\log|z-1/n|^2$, a psh function on $\CC$ (discontinuous at $z=0$). Let $p_i(z,w)=z$ and $p_j(z,w)=z+w$. Then one asks if
\[
(1-\eta)|w|^c\geq \mathrm{e}^{-\eta}\prod_{n=1}^\infty|z-1/n|^{2/n^2}-\prod_{n=1}^\infty|z+w-1/n|^{2/n^2}
\]
holds in a neighborhood of $(0,0)$. The answer is negative, unfortunately. One sees this by taking the sequence of points $(z=0,w=1/n)$. Therefore, if one wants to remove the continuity assumption on $\mathrm{e}^\varphi$ in our extension theorem, a different approach is needed.
\end{remark}

Finally, we state a direct consequence of Theorem~\ref{thm:CT-thm-refined'} that is enough for our later use.

\begin{theorem}
\label{thm:CT-thm-refined''} Let $V\subsetneq X$ be a connected positive dimensional compact complex submanifold of $(X,\omega)$. Let $T=\omega|_V+\ddc\varphi$ be a K\"ahler current with gentle analytic singularities on $V$. Then one can find a K\"ahler current $\tilde T=\omega+\ddc\tilde \varphi$ on $X$ such that 
 \begin{enumerate}
     \item $\tilde \varphi|_V=\varphi$, i.e., $\tilde T$ extends $T$;
     \item $\tilde \varphi$ has analytic singularities on $X$;
     \item $\tilde \varphi$ is continuous on $X\setminus V$, so $\tilde \varphi$ and $\varphi$ have the same singular locus.
 \end{enumerate}
\end{theorem}

\begin{proof}
Using Definition~\ref{def:analy-sing}(1), we see that $\mathrm{e}^{\varphi}\in C^{\alpha}(V)$ for some $\alpha>0$. 
So the result follows directly from Theorem~\ref{thm:CT-thm-refined'}.
\end{proof}

\section{The volume identity}
\label{sec:vol-eq}
Unless specified otherwise, in this section $\xi$ is a  big class on $X$ and $\dim X = n$. We also fix a smooth flag $Y_{\bullet}=(X=Y_0\supset \cdots \supset Y_n)$ on $X$.
Define
\[
\nu_{\min}(\xi,Y_1)\coloneqq \inf_{T\in\Omega(X,\xi)}\nu(T,Y_1),\quad \nu_{\max}(\xi,Y_1)\coloneqq \sup_{T\in\Omega(X,\xi)}\nu(T,Y_1).
\]

\begin{lemma}\label{lma:nu1closure}
	One has $\nu_{\min}(\xi,Y_1)<\nu_{\max}(\xi,Y_1)$. Moreover,
 \[
 \overline{\nu_1(\Omega(X,\xi))}= [\nu_{\min}(\xi,Y_1),\nu_{\max}(\xi,Y_1)],
 \]
 where $\nu_1$ is the first component of the valuative map $\nu$.
\end{lemma}

\begin{proof}
	Pick any $T\in\Omega(X,\xi)$, then $T-\nu_{\min}(\xi,Y_1)[Y_1]$ is a K\"ahler current, which implies that $\xi-\nu_{\min}(\xi,Y_1)\{Y_1\}$ is a big class. So for small $\varepsilon>0$, $\xi-(\nu_{\min}(\xi,Y_1)+\varepsilon)\{Y_1\}$ is a big class as well. So there exists $S\in\Omega(X,\xi)$ with $\nu(S,Y_1)\geq\nu_{\min}(\xi,Y_1)+\varepsilon$, giving $\nu_{\min}(\xi,Y_1)<\nu_{\max}(\xi,Y_1)$. The second assertion follows from the convexity of $\Omega(X,\xi)$.
\end{proof}

\begin{proposition}
\label{prop:vol-of-xi=int-from-mu-to-tau}
	Assume that $n\geq2$. One has 
 \[
 \vol(\xi)=n\int_{\nu_{\min}(\xi,Y_1)}^{\nu_{\max}(\xi,Y_1)}\vol_{X|{Y_1}}(\xi-t\{Y_1\})\,\mathrm d t.
 \]
\end{proposition}

\begin{proof}
We first observe that for any $t\in(\nu_{\min}(\xi,Y_1),\nu_{\max}(\xi,Y_1))$, one has 
\[
Y_1\cap K(\xi-t\{Y_1\})\neq\varnothing. 
\]
In fact, by Lemma~\ref{lma:nu1closure}, for $t\in(\nu_{\min}(\xi,Y_1),\nu_{\max}(\xi,Y_1))$ lying in a dense subset, there exists $T\in\Omega(X,\xi)$ with $\nu(T,Y_1)=t$. So $T-t[Y_1]$ is a K\"ahler current in $\xi-t\{Y_1\}$ with $Y_1\nsubseteq E_+(T-t[Y_1])$. 

Clearly, $\vol(\xi)=\vol(\xi-\nu_{\min}(\xi,Y_1)\{Y_1\})$. By the continuity of $\vol(\cdot)$ \cite[Proposition 4.7]{Bou02b}, one thus has
	\[
 \vol(\xi)=\lim_{\varepsilon\searrow 0}\vol(\xi-(\nu_{\min}(\xi,Y_1)+\varepsilon)\{Y_1\}).
 \]
	We can apply Theorem~\ref{thm:WN-volume-identity} to the big class $\xi-(\nu_{\min}(\xi,Y_1)+\varepsilon)\{Y_1\}$ to get
	\begin{align*}
    \vol(\xi)&=\lim_{\varepsilon\searrow 0}n\int_{\nu_{\min}(\xi,Y_1)+\varepsilon}^{\nu_{\max}(\xi,Y_1)}\vol_{X|{Y_1}}(\xi-t\{Y_1\})\,\mathrm d t\\ &=n\int_{\nu_{\min}(\xi,Y_1)}^{\nu_{\max}(\xi,Y_1)}\vol_{X|{Y_1}}(\xi-t\{Y_1\})\,\mathrm d t
	\end{align*}
    as desired.
\end{proof}

Now for any $t\in(\nu_{\min}(\xi,Y_1),\nu_{\max}(\xi,Y_1))$, we look at the following restricted volume
\[
\vol_{X|{Y_1}}(\xi-t\{Y_1\})=\sup_{T\in\Omega(X,\xi-t\{Y_1\})}\int_{Y_1}T|_{Y_1}^{n-1}
\]
and the $t$-slice of the Okounkov body
\[
\Delta_t(\xi)\coloneqq \{x\in\RR^{n-1}:(t,x)\in\Delta(\xi)\}.
\]
We shall prove that they are related in the following way (cf. \cite[Corollary 4.25(i)]{LM09}).

\begin{namedtheorem}[$C_n$]
\label{conj:restricted-vol-equality}
    If $n\geq2$, then for any $t\in(\nu_{\min}(\xi,Y_1),\nu_{\max}(\xi,Y_1))$, one has
    \[
    \vol_{X|Y_1}(\xi-t\{Y_1\})=(n-1)!\vol_{\RR^{n-1}}(\Delta_t(\xi)).
    \]
\end{namedtheorem}

\begin{lemma}\label{lma:CnimplyAn}
Assume that $n\geq 2$.
	Property~\hyperref[conj:restricted-vol-equality]{$C_n$} implies Property~\hyperref[conj:vol-equality]{$A_n$}.
\end{lemma}

\begin{proof}
	By Lemma~\ref{lma:nu1closure}, we can express $\vol_{\RR^n}(\Delta(\xi))$ as
	\[
	\vol_{\RR^n}(\Delta(\xi))=\int_{\nu_{\min}(\xi,Y_1)}^{\nu_{\max}(\xi,Y_1)}\vol_{\RR^{n-1}}(\Delta_t(\xi))\,\mathrm dt.
	\]
	So we conclude from Proposition~\ref{prop:vol-of-xi=int-from-mu-to-tau}.
\end{proof}

Our goal is now to show that for $n\geq 2$, Property~\hyperref[conj:relative-Oko-vol-equal]{$B_{n-1}$} implies Property~\hyperref[conj:restricted-vol-equality]{$C_n$}. Since we already know by Lemma~\ref{lma:rel-vol-conj=vol-conj} that Property~\hyperref[conj:vol-equality]{$A_n$} and Property~\hyperref[conj:relative-Oko-vol-equal]{$B_n$} are equivalent, this will together with Lemma~\ref{lma:CnimplyAn} complete the induction argument and thus prove Property~\hyperref[conj:vol-equality]{$A_n$} for all $n$.

Note that both maps $t\mapsto\vol_{X|Y_1}(\xi-t\{Y_1\})$ and $t\mapsto\vol_{\RR^{n-1}}(\Delta_t(\xi))$ are continuous: the former continuity is due to \cite[Corollary~4.11]{Mat13} and the latter follows from the Brunn--Minkowski inequality. Therefore, to establish Property~\hyperref[conj:restricted-vol-equality]{$C_n$} it is enough to show that 
\[
\vol_{X|Y_1}(\xi-t\{Y_1\})=(n-1)!\vol_{\RR^{n-1}}(\Delta_t(\xi))
\]
for all rational $t\in(\nu_{\min}(\xi,Y_1),\nu_{\max}(\xi,Y_1))$.

For $t\in(\nu_{\min}(\xi,Y_1),\nu_{\max}(\xi,Y_1))\cap \QQ$, we introduce $t$-slices of $\Omega(X,\xi)$ and $\mathcal A(X,\xi)$ and show how they are connected to the $t$-slice of the Okounkov body:
\[
\Omega_t(X,\xi)\coloneqq \{T\in\Omega(X,\xi): \nu(T,Y_1)= t\},
\]
\[
\mathcal{A}_t(X,\xi)\coloneqq \{T\in\mathcal{A}(X,\xi):\nu(T,Y_1)= t\}.
\]

\begin{lemma}
\label{lem:Omega-t-dense-in-Delta-t}
	For any $t\in(\nu_{\min}(\xi,Y_1),\nu_{\max}(\xi,Y_1))\cap\QQ$, one has
\[
\{t\}\times\Delta_t(\xi)=\overline{\nu(\Omega_t(X,\xi))}=\overline{\nu(\mathcal{A}_t(X,\xi))}. 
\]
\end{lemma}

\begin{proof}
It is clear that
\[
\overline{\nu(\Omega_t(X,\xi))}\subseteq \overline{\nu(\mathcal{A}_t(X,\xi))}\subseteq \{t\}\times\Delta_t(\xi).
\]
To finish we will argue that $\{t\}\times\Delta_t(\xi) \subseteq \overline{\nu(\Omega_t(X,\xi))}$. Let $a\in\{t\}\times\Delta_t(\xi)$. Then there exists a sequence $T_i\in\Omega(X,\xi)$ such that $\nu(T_i)\to a$. So in particular, $t_i\coloneqq \nu(T_i,Y_1)\to t$. Pick $S,G\in\Omega(X,\xi)$ such that $s\coloneqq \nu_1(S,Y_1)<t$ and $g\coloneqq \nu(G,Y_1)>t$. Define	\begin{equation*}
		S_i\coloneqq 
		\begin{cases}
			\frac{t-s}{t_i-s}T_i+\frac{t_i-t}{t_i-s}S,\ &\text{if }t_i\geq t,\\
			\frac{g-t}{g-t_i}T_i+\frac{t-t_i}{g-t_i}G,\ &\text{if }t_i< t.
		\end{cases}
	\end{equation*}
	Then $S_i\in\Omega_t(X,\xi)$ for $i\gg0$ and $\lim_i\nu(S_i)=\lim_i \nu(T_i)=a$, hence $a\in\overline{\nu(\Omega_t(X,\xi))}$.
\end{proof}

We are now ready to prove that Property~\hyperref[conj:relative-Oko-vol-equal]{$B_{n-1}$} implies Property~\hyperref[conj:restricted-vol-equality]{$C_n$}:

\begin{lemma}
\label{lma:A-n-1=>C-n}
	Assume that $n\geq 2$, then Property~\hyperref[conj:relative-Oko-vol-equal]{$B_{n-1}$} implies Property~\hyperref[conj:restricted-vol-equality]{$C_n$}.
\end{lemma}

\begin{proof}
Pick any $t\in(\nu_{\min}(\xi,Y_1),\nu_{\max}(\xi,Y_1))\cap\QQ$.
	Put $\xi_t\coloneqq \xi-t\{Y_1\}$ for simplicity.
It suffices to show the following:
	there is an increasing sequence $R_i\in\Omega(X,\xi_t)$ such that
	\begin{enumerate}[(1)]
        \item $E_+(R_i)=E_{nK}(\xi_t)$ for all $i$;
		\item $R_i|_{Y_1}$ has positive mass and $\int_{Y_1}R_i|_{Y_1}^{n-1}\nearrow\vol_{X|{Y_1}}(\xi_t)$;
      \item One has $\int_{Y_1}R_i|_{Y_1}^{n-1}=(n-1)!\vol_{\RR^{n-1}}(\Delta_{t,i})$, where $\Delta_{t,i}$ is the $(n-1)$-dimensional partial Okounkov body $\Delta(\xi_t|_{Y_1};R_i|_{Y_1})$ of $\xi_t|_{Y_1}$ with respect to $R_i|_{Y_1}$, which is defined using the truncated flag $Y_1\supset Y_2\supset\dots\supset Y_n$;
		\item $\Delta_{t,i}$ is an increasing sequence of $(n-1)$-dimensional convex bodies contained in $\Delta_t(\xi)$ such that $\vol_{\RR^{n-1}}(\Delta_{t,i})\nearrow\vol_{\RR^{n-1}}(\Delta_t(\xi))$.
	\end{enumerate}

The above  conditions would imply  $\vol_{X|Y_1}(\xi-t\{Y_1\})$ to equal $(n-1)!\vol_{\RR^{n-1}}(\Delta_t(\xi))$ for any $t\in(\nu_{\min}(\xi,Y_1),\nu_{\max}(\xi,Y_1))\cap\QQ$, hence proving Property~\hyperref[conj:restricted-vol-equality]{$C_n$}.

To construct the sequence $R_i$, first note that the set $\nu_{Y_\bullet}(\Omega_t(X,\xi))\subset\QQ^n_{\geq 0}$ is countable. So one can write $\nu(\Omega_t(X,\xi))=\{a_1,a_2,\dots,a_i,\dots\}$ with $a_i\in\QQ^n_{\geq 0}$. Let $a_i'\in\QQ^{n-1}_{\geq 0}$ be such that $a_i=(t,a_i')$. Pick $T_i\in\Omega_t(\xi)$ so that $\nu_{Y_\bullet}(T_i)=a_i$. Fix $R_0\in\Omega(X,\xi_t)$ so that $E_+(R_0)=E_{nK}(\xi_t)$. Also pick a sequence $G_i\in\Omega(X,\xi_t)$ so that $\int_{Y_1}G_i^{n-1}\to\vol_{X|Y_1}(\xi_t)$. Applying Lemma~\ref{lem:sing-decrease} inductively, we can find $R_i\in\Omega(X,\xi_t)$ that is less singular than $T_i-t[Y_1]$, $G_i$ and $R_{i-1}$. So $R_i\in\xi_t$ is an increasing sequence of K\"ahler currents with analytic singularities precisely on $E_{nK}(\xi_t)$, showing (1). That (2) holds is due to \cite[Theorem 1.2]{WN19b} and (3) holds thanks to Property~\hyperref[conj:vol-equality]{$B_{n-1}$}. It remains to argue (4).

Let $\nu'_{Y_\bullet}:\mathcal{A}(Y_1,\xi_t|_{Y_1})\to\QQ^{n-1}_{\geq 0}$ denote the valuative map constructed using the flag $Y_1\supset Y_2\supset\dots\supset Y_n$ on $Y_1$. Then one has that $a'_i\in\nu'_{Y_\bullet}(\mathcal{A}(Y_1,\xi_t|_{Y_1};R_i|_{Y_1}))$, since $(T_i-t[Y_1])|_{Y_1}\in \mathcal{A}(Y_1,\xi_t|_{Y_1};R_i|_{Y_1})$ and $a_i'=\nu'_{Y_\bullet}((T_i-t[Y_1])|_{Y_1})$. Thus, $a_i'\in\Delta_{t,i}$ by Lemma~\ref{lma:pob_loc_analytic}, which implies that 
$\Delta_t(\xi)\subseteq\overline{\bigcup_i\Delta_{t,i}}$ by Lemma~\ref{lem:Omega-t-dense-in-Delta-t}. 

Next, we show that $\Delta_{t,i}\subseteq\Delta_t(\xi)$ holds for all $i$. By Theorem~\ref{thm:liftflaggeneral}, there are modifications $\pi_i:X_i\rightarrow X$ such that 
    \begin{enumerate}[(i)]
        \item $\pi_i^*R_i=\beta_i+[D_i]$, where $\beta_i$ is a smooth semipositive form on $X_i$ and $D_i$ is an effective $\QQ$-divisor on $X_i$;
        \item The flag $Y_{\bullet}$ admits a lifting $(Y_\bullet^{(i)},g_i^{-1})$ to $X_i$.
    \end{enumerate}

Note that $Y^{(i)}_1$ is the strict transform of $Y_1$, and it is not contained in $D_i$ (as $\nu(R_i,Y_1)=0$). 
By the Lelong--Poincar\'e formula, the restricted current
$
[D_i]|_{Y_1^{(i)}}
$
only has divisorial components in its Siu decomposition.
Let $\nu'_{Y^{(i)}_\bullet}$ be the $\QQ^{n-1}_{\geq 0}$-valued map determined by the flag $Y_1^{(i)}\supset\dots\supset Y^{(i)}_n$ on $Y_1^{(i)}$. By Lemma~\ref{lma:flagvaluationchange}, there is  $h_i\in \mathrm{SL}_{n-1}(\ZZ)$ such that 
\[
\nu'_{Y_\bullet^{(i)}}(\pi_i^*T)h_i=\nu'_{Y_\bullet}(T)\text{ for all }T\in\Omega(Y_1,\xi_t|_{Y_1}).
\]
From the proof of Lemma~\ref{lma:flagvaluationchange}, we see that $h_i$ is obtained from $g_i$ with the first row and the first column removed.

Now, for any $T\in\Omega(Y_1,\xi_t|_{Y_1};R_i|_{Y_1})$, Lemma~\ref{lem:Siu-deomp-lem} implies that $S_i\coloneqq \pi^*_i T-[D_i]|_{Y_1^{(i)}}$ is a positive current on $Y^{(i)}_1$ in the semipositive class $\{\beta_i|_{Y_1^{(i)}}\}$ and $S_i$ has analytic singularities. 

Fix a smooth K\"ahler form $\omega_i$ on $X_i$. Then for any $\varepsilon>0$, $S_{i,\varepsilon}\coloneqq S_i+\varepsilon\omega_i|_{Y_1^{(i)}}$ is a K\"ahler current on $Y^{(i)}_1$ in the K\"ahler class $(\beta_i+\varepsilon\omega_i)|_{Y^{(i)}_1}$ with analytic singularities. Note that (merely from Definition~\ref{def:analy-sing}) it is unclear if $S_{i,\varepsilon}$ has gentle analytic singularities on $Y^{(i)}_1$ or not, but thanks to Theorem~\ref{thm:Dem-reg} one can approximate $S_{i,\varepsilon}$ by a sequence, say $\{S_{i,\varepsilon,k}\}_{k\in\NN}$, with gentle analytic singularities on $Y^{(i)}_1$.

Applying Theorem~\ref{thm:CT-thm-refined''}, we get  $\tilde S_{i,\varepsilon,k}\in\{\beta_i+\varepsilon\omega_i\}$, currents on $X_i$, extending  $S_{i,\varepsilon,k}$ for each $k$. Moreover, $\tilde S_{i,\varepsilon,k}$ has analytic singularities on $X_i$. Then (the proof of) Lemma~\ref{lem:omega-dense-in-gamma} implies that
\begin{align*}
(0,\nu'_{Y^{(i)}_\bullet}(S_{i}))=(0,\nu'_{Y^{(i)}_\bullet}(S_{i,\varepsilon}))&=(0,\lim_{k\to\infty}\nu'_{Y^{(i)}_\bullet}(S_{i,\varepsilon,k})) \\&=\lim_{k\to \infty}\nu_{Y^{(i)}_\bullet}(\tilde S_{i,\varepsilon,k})\in\Delta_{Y^{(i)}_\bullet}(\{\beta_i+\varepsilon\omega_i\})
\end{align*}
for any $\varepsilon>0$. By Proposition~\ref{prop:body-properties}(3) we can let $\varepsilon\to 0$ to find that 
\[
\big(0,\nu'_{Y^{(i)}_\bullet}(\pi^*_i T)\big)\in\Delta_{Y_\bullet^{(i)}}(\{\beta_i\})\cap\{0\}\times\RR^{n-1}+\nu_{Y^{(i)}_\bullet}([D_i]),
\]
where we used  $(0,\nu'_{Y^{(i)}_\bullet}([D_i]|_{Y_1^{(i)}}))=\nu_{Y_\bullet^{(i)}}([D_i])$.
Since $ (0,\nu'_{Y_\bullet}(T))=(0,\nu'_{Y^{(i)}_\bullet}(\pi^*_i T)h_i)=(0,\nu'_{Y^{(i)}_\bullet}(\pi^*_i T))g_i$, we have that 
\begin{equation*}
        (0,\nu'_{Y_\bullet}(T))\in \big(\Delta_{Y_\bullet^{(i)}}(\{\beta_i\})\cap\{0\}\times\RR^{n-1}+\nu_{Y^{(i)}_\bullet}([D_i])\big)g_i.
\end{equation*}
Thus, we see that
\[
\{0\}\times\Delta_{t,i}\subseteq\big(\Delta_{Y_\bullet^{(i)}}(\{\beta_i\})\cap\{0\}\times\RR^{n-1}+\nu_{Y^{(i)}_\bullet}([D_i])\big)g_i.
\]

On the other hand, for any $G\in\Omega(X_i,\{\beta_i\})$ with $\nu(G,Y_1^{(i)})=0$, $G+[D_i]$ is an element in $\Omega(\pi^*_i\xi_t)$. One can then apply Proposition~\ref{prop:approx-push-forward-current} to find a sequence $G_k\in\mathcal{A}(X,\xi_t)$ such that
\begin{align*}
        \nu_{Y^{(i)}_\bullet}(G+[D_i])&=\lim_{k\to\infty}\nu_{Y^{(i)}_\bullet}(\pi^*_i G_k)\\ &=\lim_{k\to\infty}\nu_{Y_\bullet}(G_k)g_i^{-1}\in \big(\{0\}\times\Delta_t(\xi)\big)g_i^{-1},
\end{align*}
where the last inclusion follows from Lemma~\ref{lem:Omega-t-dense-in-Delta-t} (since $G_k+t[Y_1]\in\mathcal{A}_t(X,\xi)$, as $t \in \mathbb{Q}$). Thus,
\[
\big(\Delta_{Y_\bullet^{(i)}}(\{\beta_i\})\cap\{0\}\times\RR^{n-1}+\nu_{Y^{(i)}_\bullet}([D_i])\big)g_i\subseteq\{0\}\times\Delta_t(\xi).
\]
Therefore,
$
\Delta_{t,i}\subseteq\Delta_t(\xi)
$
follows, as claimed.

From the above discussions we conclude that $\Delta_t(\xi)=\overline{\bigcup_i\Delta_{t,i}}$.
Moreover, since $R_i$ is an increasing sequence, we see that $\Delta_{t,i}$ is an increasing sequence of $(n-1)$-dimensional convex bodies contained in $\Delta_t(\xi)$. So we deduce that $\vol_{\RR^{n-1}}(\Delta_{t,i})\nearrow\vol_{\RR^{n-1}}(\Delta_t(\xi))$, concluding the proof of (4).
\end{proof}

We summarize our findings in the next result that proves both Theorem~\ref{thm:main} and Corollary~\ref{cor:main}:

\begin{theorem}\label{thm:volumeidentity}
Let $\xi$ be a big cohomology class on $X$. Let $Y_\bullet$ be a flag on $X$. Then one can associate to it a natural compact convex set $\Delta_{Y_\bullet}(\xi)$ with $\xi$ so that 
\[
\vol_{\RR^n}(\Delta_{Y_\bullet}(\xi))=\frac{1}{n!}\vol(\xi).
\]

\end{theorem}

\begin{proof}Let $n=\dim X$. Let us first assume that $\xi$ is big. We will show that Property~\hyperref[conj:vol-equality]{$A_n$} and Property~\hyperref[conj:relative-Oko-vol-equal]{$B_n$} hold for any dimension $n\geq1$ and Property~\hyperref[conj:restricted-vol-equality]{$C_n$} holds for any $n\geq2$. 
We have proved the following implications:
\begin{enumerate}
    \item \hyperref[conj:vol-equality]{$A_n$}$\iff$\hyperref[conj:relative-Oko-vol-equal]{$B_n$} for $n\geq 1$ by Lemma~\ref{lma:rel-vol-conj=vol-conj};
    \item \hyperref[conj:relative-Oko-vol-equal]{$B_{n-1}$}$\implies$\hyperref[conj:restricted-vol-equality]{$C_n$} for $n\geq 2$ by Lemma~\ref{lma:A-n-1=>C-n};
    \item \hyperref[conj:restricted-vol-equality]{$C_n$}$\implies$\hyperref[conj:vol-equality]{$A_n$} for $n\geq 2$ by Lemma~\ref{lma:CnimplyAn}.
\end{enumerate}
By Lemma~\ref{lem: A_1}  Property~\hyperref[conj:vol-equality]{$A_1$} holds, so we conclude the big case by induction.
\end{proof}

As a simple corollary, we obtain the result for pseudoeffective classes as well:

\begin{corollary}\label{cor:volumeidentity}
Let $\xi$ be a pseudoeffective cohomology class on $X$. Let $Y_\bullet$ be a flag on $X$. Then one can associate to it a natural compact convex set $\Delta_{Y_\bullet}(\xi)$ with $\xi$ so that 
\[
\vol_{\RR^n}(\Delta_{Y_\bullet}(\xi))=\frac{1}{n!}\vol(\xi).
\]

\end{corollary}
\begin{proof} Due to the previous theorem, we can assume that $\xi$ is pseudoeffective but not big. In this case $\vol(\xi)=0$. By the definition of $\Delta_{Y_\bullet}(\xi)$ (recall \eqref{eq: psef_Okounkov}) we obtain that 
\[
\vol_{\RR^n}(\Delta_{Y_\bullet}(\xi))\leq \vol_{\RR^n}(\Delta_{Y_\bullet}(\xi + \varepsilon \{\omega\}))\to 0,
\]
thanks to the continuity of the volume functional on the pseudoeffective cone \cite{Bou02b}.
\end{proof}

\section{Moment bodies}\label{sect_moment1}

The purpose of this section is to give an alternative approach to defining a convex body associated with a big cohomology class $\xi$ on a connected compact K\"ahler manifold $X$. Instead of using valuation vectors $\nu(T)$ of currents $T\in \Omega(X,\xi)$, we consider certain iterated toric degenerations, where say $(X,T)$ degenerates to $(\mathbb{P}^n,T_n)$, $T_n$ being some closed positive toric $(1,1)$-current. 
Since $T_n$ is toric it has an associated moment body $\Delta(T_n)\subseteq \RR^n$ (as explained below). We then define the moment body $\Delta^{\mu}(\xi)$ as the closure of the union of the moment bodies $\Delta(T_n)$ of all toric degenerations. In the end we will show that this moment body actually coincides with the Okounkov body $\Delta(\xi)$, thus establishing a link between transcendental Okounkov bodies and toric degenerations. 

\paragraph{Moment bodies of toric currents on $\PP^n$.}

Let $T$ be a closed positive $(1,1)$-current on $\PP^n$ which is toric (i.e. $(S^1)^n$-invariant) with respect to the standard toric action on $\PP^n$. On $\CC^n\subset \PP^n$ we can then write $T=\ddc\phi$ where $\phi$ is a toric psh function on $\mathbb{C}^n$, and  $u_{\phi}(x_1,\dots,x_n)\coloneqq \phi(\mathrm{e}^{x_1/2},\dots,\mathrm{e}^{x_n/2})$ is a convex function on $\RR^n$.

\begin{definition} The \emph{moment body} $\Delta(T)$ of $T$ is the image of $\nabla u_{\phi}$, the subgradient map of $u_{\phi}$: 
\[
\Delta(T)\coloneqq \overline{\nabla u_{\phi}(\RR^n)}.
\]
\end{definition}

If $T$ is Kähler on $(\CC^*)^n$ then in particular $((\CC^*)^n,T)$ is a toric symplectic manifold, and $\mu(\mathrm{e}^{(x_1+iy_1)/2},\dots,\mathrm{e}^{(x_n+iy_n)/2})\coloneqq \nabla u_{\phi}(x_1,\dots,x_n)$ is the so-called \emph{moment map}, explaining our terminology (see e.g. \cite{CdS19}).

If $\{T\}=c_1(\mathcal{O}(k))$ and $\int_{\PP^n}T^n>0$, then the moment body $\Delta(T)$ coincides with the partial Okounkov body of the singular metric corresponding to $T$ in the sense of \cite{Xia21}, as proved in \cite[Theorem~8.3]{Xia21}.

Let us now record three well-known facts, whose proofs we give for the convenience of the reader.

\begin{lemma} \label{Lem:convmoment}
    The moment body $\Delta(T)$ is convex.
\end{lemma}

\begin{proof}
Let $u\coloneqq u_{\phi}$. Note that $u$ is convex and finite, and hence continuous. Let $v$ be a smooth convex function on $\RR^n$ such that $\nabla(v)(\RR^n)=B_1$, where $B_1$ denotes the unit ball. Let $\epsilon>0$ and $u_{\epsilon}\coloneqq u+\epsilon v$. Let $y_1,y_2\in \nabla u(\RR^n)$. We then get that the functions $u(x)-x\cdot y_1$ and $u(x)-x\cdot y_2$ are both bounded from below, and thus $u_{\epsilon}(x)-x\cdot y_1$ and $u_{\epsilon}(x)-x\cdot y_1$ are both proper and continuous. Given $t\in[0,1]$ it follows that 
\[
t(u_{\epsilon}(x)-x\cdot y_1)+(1-t)(u_{\epsilon}(x)-x\cdot y_2)=u_{\epsilon}(x)-x\cdot(ty_1+(1-t)y_2)
\]
also is proper and continuous, and hence achieves a minumum at some point $x_0$. One then sees that $ty_1+(1-t)y_2$ is a subgradient of $u_{\epsilon}$ at $x_0$, thus $ty_1+(1-t)y_2\in \nabla u_{\epsilon}(\RR^n)$ for all $\epsilon>0$. Now we note that $\nabla u_{\epsilon}(\RR^n)\subseteq \nabla u(\RR^n)+B_{\epsilon}$. It follows that $ty_1+(1-t)y_2\in \nabla u(\RR^n)+B_{\epsilon}$ for all $\epsilon>0$, and hence $ty_1+(1-t)y_2\in \overline{\nabla u(\RR^n)}$.

If $y_1,y_2\in \overline{\nabla u(\RR^n)}$ then $y_1$ and $y_2$ can be approximated by $y'_1,y'_2\in \nabla u(\RR^n)$, hence $ty'_1+(1-t)y'_2\in \overline{\nabla u(\RR^n)}$, and going to the limit yields that $ty_1+(1-t)y_2\in \overline{\nabla u(\RR^n)}$. This shows that $\Delta(T)=\overline{\nabla u(\RR^n)}$ is convex. 
\end{proof} 

\begin{lemma} \label{Lem:inclmoment}
Suppose that $T_1$ and $T_2$ are two toric closed positive $(1,1)$-currents on $\mathbb{P}^n$. Assume that $T_1$ is more singular than $T_2$, then $\Delta(T_1)\subseteq \Delta(T_2)$.
\end{lemma}

\begin{proof}
Let $u_1\coloneqq u_{\phi_1}$ and $u_2\coloneqq u_{\phi_2}$ and let $v$ be as in the proof of Lemma~\ref{Lem:convmoment}. For $\epsilon>0$ let $u_{1,\epsilon}\coloneqq u_1+\epsilon v$ and $u_{2,\epsilon}\coloneqq u_2+\epsilon v$.  Let $y\in \nabla u_1(\RR^n)$. It follows then that the function $u_1(x)-x\cdot y$ is bounded from below, and since $u_1\leq u_2 + C$ we get that $u_{2,\epsilon}(x)-x\cdot y$ is proper and continuous for all $\epsilon>0$. Hence, $u_{2,\epsilon}(x)-x\cdot y$ achieves its minimum at some point $x_{\epsilon}$, and there $y$ will be a subgradient of $u_{2,\epsilon}$. Thus, $y\in \nabla u_2(\RR^n)+B_{\epsilon}$ for all $\epsilon$, and hence $y\in \overline{\nabla u_2(\RR^n)}$. We have thus shown that $\overline{\nabla u_1(\RR^n)}\subseteq \overline{\nabla u_2(\RR^n)}$, 
 i.e., $\Delta(T_1)\subseteq \Delta(T_2)$.
\end{proof}

\begin{proposition} \label{prop:momcureq}Under the assumptions as above,
 \[
 \vol_{\RR^n}(\Delta(T))=\frac{1}{n!}\int_{\PP^n}T^n.
 \]
\end{proposition}

\begin{proof}
We have that 
\begin{align*}
    \int_{\PP^n}T^n=\int_{(\CC^*)^n}(\ddc\phi)^n=n!\int_{\RR^n}\MA(u_{\phi})&=n!\vol_{\RR^n}(\nabla u(\RR^n))\\&=n!\vol_{\RR^n}(\Delta(T)),
\end{align*}
where $\MA$ denotes the real Monge--Amp\`ere operator in the sense of Alexandrov.
In the second equality we used the standard relationship between the complex and the real Monge--Amp\`ere operator \cite[Lemma~2.2]{CGSZ19} (note that the normalization of the real Monge--Amp\`ere operator in \cite{CGSZ19} differs from the usual convention by $n!$) while the third equality used the definition of the real Monge--Amp\`ere operator, namely that 
\[
\int_A \MA(u)\coloneqq \vol_{\RR^n}(\nabla u(A))
\]
for any Borel subset $A\subseteq \mathbb{R}^n$.
\end{proof}

\paragraph{Deformation to the normal bundle.}

Recall that if $Y$ is a connected submanifold of a complex manifold $X$, the blow-up of $X\times \mathbb{P}^1$ along $Y\times\{0\}$ defines the total space $\mathcal{X}$ of what is called the (standard) deformation of $X$ to the normal bundle $N_{Y|X}$ of $Y$. There is an obvious morphism $\mathcal{X}\rightarrow \mathbb{P}^1$.
The zero fiber has two components: $Z$ which is isomorphic to the blow-up of $Y$ in $X$, and the exceptional divisor $\mathcal{E}$, which is isomorphic to the projective completion of $N_{Y|X}$. The natural $\mathbb{C}^*$-action on $X\times \mathbb{P}^1$ (i.e. $\tau(z,w)\coloneqq (z,\tau w)$) lifts to $\mathcal{X}$, and it restricts to the inverse $\mathbb{C}^*$-action on the normal bundle (i.e. $\tau$ acts on a fiber by multiplication by $\tau^{-1}$).

\paragraph{Toric degenerations.}

Now assume that $Y_{\bullet}$ is a smooth flag in $X$. Let $\cX^1$ denote the deformation of $X$ to the normal bundle $N_{Y_1|X}$ of $Y_1$, and let $X^1$ denote the exceptional divisor, so that $X^1$ is ismorphic to $\PP(N_{Y_1|X}\oplus \CC)$. Let also $Z^1$ denote the other component of the zero fiber.

If we use the natural identification between $Y_1$ and the zero section of $N_{Y_1|X}$ we can identify $Y_2$ with a submanifold in $X^1$. Thus, we let $\cX^2$ denote the deformation of $X^1$ to the normal bundle $N_{Y_2|X^1}$ of $Y_2$. Also let $X^2$ denote the exceptional divisor, so that $X^2$ is isomorphic to $\PP(N_{Y_2|X^1}\oplus \CC)$. We also let $Z^2$ denote the other component of the zero fiber. As $Y_2$ is fixed by the $\CC^*$-action on $X^1$, the action extends to $\cX^2$. This gives $\cX^2$ a combined $(\CC^*)^2$-action, which also restricts to $X^2$.

Continuing this way we get a sequence of deformation spaces $\cX^i$ for $i=1,\dots,n$, with generic fiber $X^{i-1}$ (with $X^0\coloneqq X$) and exceptional divisor $X^i$ (the other component of the zero fiber is denoted by $Z^i$), such that $\cX^i$ has a $(\CC^*)^i$-action restricting to $X^i$.

Choose local holomorphic coordinates $z_i$ on $X$ centered at $Y_n=x\in X$ such that locally $Y_i=\{z_1=\dots=z_i=0\}$. This induces local holomorphic coordinates $z_i$ on $N_{Y_1|X}$ centered at $Y_n=x\in N_{Y_1|X}$ such that the $\CC^*$-action on the normal bundle acts by $\tau(z_1,\dots,z_n)=(\tau z_1,z_2,\dots,z_n)$. When taking the normal bundle of $Y_2$ we again get an induced set of local holomorphic coordinates $z_i$ centered at $Y_n=x\in N_{Y_2|X^1}$. The $\CC^*$-action inherited from $N_{Y_1|X}$ will still act by $\tau(z_1,\dots,z_n)=(\tau z_1,z_2,\dots,z_n)$, while the new $\CC^*$-action coming from the normal bundle structure will be $\tau(z_1,\dots,z_n)=(\tau z_1,\tau z_2, z_3,\dots,z_n)$. Continuing this way we see that we get local coordinates $z_i$ on each $X^i$ centered at $Y_n=x\in X^i$ adapted to the induced $(\CC^*)^i$-action. 

In the last step we consider the normal bundle of $Y_n=x\in X^{n-1}$, but since $x$ is just a point, $N_{Y_n|X^{n-1}}$ is the tangent space $T_x(X^{n-1})$ of $x$ in $X^{n-1}$. One can see that the induced local coordinates $z_i$ on $T_x(X^{n-1})$ respect the vector space structure, giving us an identification of $T_x(X^{n-1})$ with $\CC^n$ which adapts to the torus action. As $X^n=\PP(N_{Y_n|X^{n-1}}\oplus \CC)$ we see that $X^n \simeq \PP^n$, and we will identify $X^n$ with $\PP^n$ by means of the identification of $T_x(X^{n-1})$ with $\CC^n$. The induced $(\CC^*)^n$-action on $X^n$, described above, acts on $\CC^n\subseteq X^n$ by $(\tau_1,\dots,\tau_n)(z_1,\dots,z_n)=(\tau_1\tau_2\dots\tau_n z_1,\tau_2\tau_3\dots\tau_n z_2,\dots,\tau_n z_n)$. Here we note that this is different from the standard action $(\tau_1,\dots,\tau_n)(z_1,\dots,z_n)=(\tau_1 z_1,\tau_2 z_2,\dots,\tau_n z_n)$. But we note that a current is toric (i.e. $(S^1)^n$-invariant) with respect to the induced action if and only if it is toric with respect to the standard action.

Via the above method $X$ is degenerated (in $n$ steps) to $\PP^n$, without taking account of the class $\xi$. To remedy this, we want to equip the deformation spaces $\cX^i$ with compatible currents, that somehow `metrize' the $n$-step degeneration (in a possibly singular sense).

Note that there is a natural morphism $\mathcal{X}^i\rightarrow \mathbb{P}^1$ for each $i=1,\ldots,n$. For any $c\in \mathbb{P}^1$, we denote the fiber of $\mathcal{X}^i$ over $c$ as $\mathcal{X}^i_c$.
For $i=1,\ldots,n-1$, let $\pi_{X^{i}}:\mathcal{X}^{i+1}\rightarrow X^i$ denote the natural morphism. We write $\pi_{X^{0}}=\pi_X\colon \mathcal{X}^1\rightarrow X$ for the natural projection map. Note that the inclusion $X\hookrightarrow X\times \mathbb{P}^1$ as the fiber over $1$ induces an isomorphism $X\cong \mathcal{X}^1_1$. We will fix this isomorphism and identify $X$ with $\mathcal{X}^1_1$. Similarly, we identify $X^{i-1}$ with $\mathcal{X}^{i}_1$ in the natural way for any $i=2,\ldots,n$.

\begin{definition}\label{def:toricdeg}
A choice of closed positive $(1,1)$-currents $\cT_i$ on $\cX^i$ for each $i=1,\ldots,n$ is called a \emph{toric degeneration} of $(X,\xi)$ (with respect to the flag $Y_{\bullet}$) if: 

\begin{enumerate}
    \item 
$\cT_i$ is $(S^1)^i$-invariant for all $i=1,\ldots,n$;
\item  for all $i=1,\ldots,n$, 
$\cT_i$ has smooth local potentials on the complement of some analytic set $A_i\subset \mathcal{X}^i$ which does not contain (but can intersect) $\cX^i_1$ or $X^i$;
\item ${\cT_1}|_{\cX^1_1}\in \mathcal{Z}_+(X,\xi)$ and  ${\cT_i}|_{\cX^{i}_1}={\cT_{i-1}}|_{X^{i-1}}$ for $i=2,\ldots,n$.
\end{enumerate}

\end{definition}
Note that by the second assumption we are allowed to restrict $\cT_i$ to $\cX^i_1$ and $X^i$.

If we denote by $T_0\coloneqq {\cT_1}|_{\cX^1_1}$ and $T_i\coloneqq {\cT_i}|_{X^i}$ for $i=1,\ldots,n$, we see that the toric degeneration $\cT_{\bullet}$ encodes a degeneration of $(X,T_0)$ to $(\PP^n,T_n)$, and that $T_n$ is a toric current. 
By abuse of language, we will thus also  call $(\PP^n,T_n)$ a \emph{toric degeneration} of $(X,\xi)$. It should be understood that when we refer to a toric degeneration $(\PP^n,T_n)$, the currents $\cT_1,\ldots,\cT_n$ are implicitly chosen. We also say the $\cT_{\bullet}$'s are the \emph{family of currents} associated with the toric degeneration $(\PP^n,T_n)$.

\paragraph{The moment body of $(X,\xi)$.}

We now note that if $(\PP^n,T_n)$ is a toric degeneration of $(X,\xi)$, then $T_n$ will have an associated moment body $\Delta(T_n)$.

\begin{definition}
    We define the moment body $\Delta^{\mu}(\xi)=\Delta^{\mu}_{Y_{\bullet}}(\xi)$ of $\xi$ (with respect to the flag $Y_{\bullet}$) as the closure of the union of the moment bodies $\Delta(T_n)$ of all toric degenerations $(\PP^n,T_n)$ of $(X,\xi)$ (with respect to the flag $Y_{\bullet}$).
\end{definition}

From the definition it is not obvious that the moment body is convex, since the union of convex sets need not be convex. However, the convexity will follow from the next proposition.

\begin{proposition} \label{prop:momunion}
    If $(\PP^n,T_n)$ and $(\PP^n,T'_n)$ are two toric degenerations of $(X,\xi)$ one can find a third toric degeneration $(\PP^n,T''_n)$ of $(X,\xi)$ such that
    \[
    \Delta(T''_n)\supseteq \Delta(T_n)\cup \Delta(T'_n).
    \]
\end{proposition}

The proof of Proposition~\ref{prop:momunion} will rely on the following lemma.

\begin{lemma} \label{lem:momnewdeg}
Let $\cT_{\bullet}$ be a toric degeneration of $(X,\xi)$. For any $i=1,\ldots,n$ and $b,c\in \mathbb{R}_{\geq 0}$, the currents $\tilde{\cT}_{\bullet}$ defined by 
\[
\tilde{\cT}_j\coloneqq 
\left\{ 
\begin{aligned}
    \cT_j,&\quad \text{if }j\neq i,i+1;\\
    \cT_i+b[\cX^i_{\infty}]+c[Z^i],&\quad \text{if }j=i;\\
    \cT_{i+1}+c[\pi_{X^i}^*([Z^i]_{|X^i})],& \quad \text{if }j=i+1\leq n
\end{aligned}
\right.
\]
also give a toric degeneration of $(X,\xi)$. Furthermore, we have that 
\[
\Delta(\tilde{T}_n)=\Delta(T_n).
\]
\end{lemma}

\begin{proof}
We first verify that $\tilde{\cT}_{\bullet}$ is a toric degeneration. It suffices to verify the three conditions in Definition~\ref{def:toricdeg}. Condition~1 follows from the facts that $\mathcal{T}_i$, $[\cX^i_{\infty}]$ and $[Z^i]$ are all $(S^1)^i$-invariant and that $\pi_{X^i}$ is $\mathbb{C}^*$-equivariant. Condition~2 and Condition~3 are clear by definition.

To see that the moment body remains unchanged, note that if $i<n-1$, then $\tilde{T}_n=T_n$, in which case the equality is obvious. If $i=n-1$, note that $Z^{n-1}$ does not intersect the zero section of $X^{n-1}$ which we identify with $Y_{n-1}$. From this it follows that $\pi_{X^{n-1}}^{-1}(X^{n-1}\cap Z^{n-1})$ does not intersect $X^n=\PP^n$. Since $\pi_{X^i}^*([Z^i]_{X^i})$ has support on $\pi_{X^{n-1}}^{-1}(X^{n-1}\cap Z^{n-1})$ we thus get that also in this case $\tilde{T}_n=T_n$. Now assume that $i=n$. Here we see that $\tilde{T}_n=T_n+c[Z_n]|_{\PP^n}$, but since $\PP^n\cap Z^n$ is the divisor at infinity and the moment body only depends on $\tilde{T}_n|_{\CC^n}$, we get that also in this case the moment body is unchanged.   
\end{proof}

\begin{proof}[Proof of Proposition~\ref{prop:momunion}]
    Let $\cT_{\bullet}$ and $\cT'_{\bullet}$ be the families of currents associated with the toric degenerations $(\PP^n,T_n)$ and $(\PP^n,T'_n)$. First we want to show that we can find toric degenerations $\tilde{\cT}_{\bullet}$ and $\tilde{\cT}'_{\bullet}$ such that   
    $\Delta(\tilde{T}_n)=\Delta(T_n)$ and $\Delta(\tilde{T}'_n)=\Delta(T'_n)$ and such that $\{\tilde{\cT}_i\}=\{\tilde{\cT}'_i\}$ for all $i=1,\ldots,n$.

    Let $\beta_i\coloneqq \{\cT_i\}$ and $\beta'_i\coloneqq \{\cT'_i\}$. Recall that $H^{1,1}(\cX^i,\RR)$ is generated by the pullback of $H^{1,1}(X^{i-1}\times \PP^1, \RR)$ together with the class $\{[X^i]\}$. Since ${\beta_1}|_{\cX^1_1}={\beta'_1}|_{\cX^1_1}=\xi$ we thus have that 
    \[
    \beta_1=\pi_{X^0}^*\xi+b_1\{[\cX^1_{\infty}]\}-c_1\{[X^1]\},\quad \beta'_1=\pi_{X^0}^*\xi+b'_1\{[\cX^1_{\infty}]\}-c'_1\{[X^1]\}
    \]
    for some real numbers $b_1,c_1,b_1',c_1'$. Now note that 
    \[
    \{[Z^1]\}=\{[\cX^1_0]-[X^1]\}=\{[\cX^1_0]\}-\{[X^1]\}=\{[\cX^1_{\infty}]\}-\{[X^1]\},
    \]
    so we can write 
    \[
    \beta_1=\pi_X^*\xi+(b_1-c_1)\{[\cX^1_{\infty}]\}+c_1\{[Z^1]\},
    \]
    \[
    \beta'_1=\pi_X^*\xi+(b'_1-c'_1)\{[\cX^1_{\infty}]\}+c'_1\{[Z^1]\}.
    \]
    From this it is clear that one can find positive numbers $d_1,d'_1, e_1,e'_1$ such that 
    \[
    \{\cT_1+d_1[\cX^1_{\infty}]+e_1[Z^1]\}=\{\Omega'_1+d'_1[\cX^1_{\infty}]+e'_1[Z^1]\}.
    \]
    Let $\tilde{\cT}_{\bullet}$ be defined by $\tilde{\cT}_1\coloneqq \Omega_1+d_1[\cX^1_{\infty}]+e_1[Z^1]$, $\tilde{\cT}_2\coloneqq \cT_2+e_1\pi^*_{X^1}([Z^1]_{|X^1})$ and $\tilde{\cT}_j\coloneqq \cT_j$ for $j>2$. Similarly, let $\tilde{\cT}'_1\coloneqq \cT'_1+d'_1[\cX^1_{\infty}]+e'_1[Z^1]$, $\tilde{\cT}'_2\coloneqq \cT'_2+e'_1\pi^*_{X^1}([Z^1]_{|X^1})$ and $\tilde{\cT}'_j\coloneqq \cT'_j$ for $j>2$. 
    
    It follows from Lemma~\ref{lem:momnewdeg} that $\tilde{\cT}_{\bullet}$ and $\tilde{\cT}_{\bullet}'$ are new toric degenerations such that 
    \[
    \Delta(\tilde{T}_n)=\Delta(T_n).
    \]
    We have also made sure that $\{\tilde{\cT}_1\}=\{\tilde{\cT}'_1\}$. If $\{\tilde{\cT}_i\}\neq\{\tilde{\cT}'_i\}$ for some $i>1$ we can use the same trick of modifying $\tilde{\cT}_{\bullet}$ and $\tilde{\cT}_{\bullet}'$ without changing the moment bodies, eventually achieving that $\{\tilde{\cT}_i\}=\{\tilde{\cT}'_i\}$ for all $i$.  

    Thus, without loss of generality we can assume that $\{\cT_i\}=\{\cT_i'\}$ for all $i=1,\ldots,n$. For all such $i$, pick a smooth closed real $(S^1)^i$-invariant $(1,1)$-form $\theta_i\in \{\cT_i\}$ and write $\cT_i=\theta_i+\ddc\phi_i$ and $\cT'_i=\theta_1+\ddc\phi'_i$ for some $\phi_i,\phi_i'\in \PSH(\mathcal{X}^i,\theta_1)$. Let also $\max_{\mathrm{reg}}(x,y)$ be a regularized maximum function with the properties that it is smooth and convex, 
    \[
    \max(x,y)\leq \max_{\mathrm{reg}}(x,y)\leq \max(x,y)+1,
    \]
    and such that 
    \[
    \max_{\mathrm{reg}}(x,y)=\max(x,y) \textup{ if } |x-y|\geq 1.
    \]
    
    Let $\cT''_1\coloneqq \theta_1+\ddc\max_{\mathrm{reg}}(\phi_1,\phi'_1)$. It is clear that $\cT''_1$ is $(S^1)$-invariant, has smooth local potentials on the complement of some analytic set which does not contain $\cX^1_1$ or $X^1$, and that $\{{\cT''_1}_{|\cX^1_1}\}=\xi$.

    If $\{\cT_2\}=\{\pi_{X^1}^*T_1\}+d_2\{[\cX^2_{\infty}]\}+e_2\{[Z^2]\}$, write $\pi_{X^1}^*T''_1+d_2[\cX^2_{\infty}]+e_2[Z^2]=\theta_2+\ddc\psi_2$ for some $\psi_2\in \PSH(\mathcal{X}^2,\theta_2)$. Since the current $\cT''_1$  is less singular than $\cT_1$ and $\cT'_1$ it follows that ${\psi_2}|_{\cX^2_1}$ is less singular than ${\phi_2}|_{\cX^2_1}$ and ${\phi'_2}|_{\cX^2_1}$. After possibly adding a constant we can thus assume that $\psi_2\geq\max_{\mathrm{reg}}(\phi_2,\phi'_2)+1$ on $\cX^2_1$. Now let $\cT''_2\coloneqq \theta_2+\ddc\max_{\mathrm{reg}}(\psi_2,\max_{\mathrm{reg}}(\phi_2,\phi'_2))$.

    It is clear that $\cT''_2$ is $(S^1)^2$-invariant, has smooth local potentials on the complement of some analytic set which does not contain $\cX^2_1$ or $X^2$. Also, since 
    \[
    \max_{\mathrm{reg}}(\psi_2,\max_{\mathrm{reg}}(\phi_2,\phi'_2))=\psi_2
    \]
    on $\cX^2_1$, we have that ${\cT''_2}_{|\cX^2_1}={\cT''_1}_{|X^1}$.
    
    Continuing this way we can construct currents $\cT''_i$ for all $i=1,\dots,n$, and as above one can easily check that $\cT''_{\bullet}$ will have the properties of a toric degeneration of $(X,\xi)$. We also see that the current $T''_n$ will be less singular than $T_n$ and $T'_n$, and so the desired inclusion 
    \[
    \Delta(T''_n)\supseteq \Delta(T_n)\cup \Delta(T'_n)
    \]
    follows from Lemma~\ref{Lem:inclmoment}.
\end{proof}

\begin{corollary}
    The moment body $\Delta^{\mu}(\xi)$ is convex.
\end{corollary}

\begin{proof}
    From Lemma~\ref{Lem:convmoment} and Proposition~\ref{prop:momunion} it follows that $\Delta^{\mu}(\xi)$ is the closure of a convex set, and hence it is convex. 
\end{proof}

We now prove a volume inequality for the moment body.

\begin{proposition}\label{prop:volmuOkounkovupperbound}
Under the above assumptions, we have that 
\[
\vol_{\RR^n}(\Delta^{\mu}(\xi))\leq \frac{1}{n!}\vol(\xi).
\]  
\end{proposition}

\begin{proof}
It follows from Proposition~\ref{prop:momunion} that the volume of $\Delta^{\mu}(\xi)$ is equal to the supremum of the volume of the moment bodies of the toric degenerations of $(X,\xi)$.  

Let $\cT_{\bullet}$ be a toric degeneration. We first claim that for all $i=1,\ldots,n$: 
    \begin{equation} \label{eq:momclaim}
    \int_{X^i}T_i^n\leq \int_{\cX^i_1}\cT_i^n.
    \end{equation}
    From the proof of Proposition~\ref{prop:momunion} we see that 
    $
    \{\cT_i\}=\pi_{X^{i-i}}^*\{T_{i-1}\}+b_i\{[\cX^i_{\infty}]\}-c_i\{[X^i]\}
    $
    for some real numbers $b_i,c_i$. We let
    \[
    \tilde{\cT}_i\coloneqq\left\{
    \begin{aligned}
        \cT_i-b_i[\cX^i_{\infty}]+c_i[X^i],& \quad \text{if }c_i\geq 0;\\
        \cT_i-(b_i-c_i)[\cX^i_{\infty}]-c_i[Z^i],& \quad \text{if }c_i<0.
    \end{aligned}
    \right.
    \]
     It follows that in both cases $\{\tilde{\cT_i}\}=\pi_{X^{i-i}}^*\{T_{i-1}\}$ and that $\tilde{\cT}_i$ is positive on $\cX^i\setminus \cX^i_{\infty}$. Let $\theta$ be a smooth closed real $(1,1)$-form in $\{T_{i-1}\}$. We can then write $\tilde{\cT}_i=\pi^*_{X^{i-1}}\theta+\ddc\phi_i$ where $\phi_i$ is a $\pi^*_{X^{i-1}}\theta$-psh function on $X^{i-1}\times \CC$. Since $\cT_i$ is $S^1$-invariant with respect to the action on the base, so is $\tilde{\cT}_i$, and therefore $\phi_i$ is $S^1$-invariant. For a fixed $x\in X^{i-1}$ we get that $\phi_i(x,\tau)$ is subharmonic in $\tau$, and since it also only depends on $|\tau|$ it follows that it is increasing in $|\tau|$. Since $\tilde{\cT}_i=\cT_i$ on $X^{i-1}\times \CC^*$ and using monotonicity \cite{WN19b} we get that 
     \[
     \int_{\cX^i_{\tau}}\cT_i^n=\int_{X^{i-1}}(\theta+\ddc\phi_i(\cdot,\tau))^n
     \]
     is increasing in $|\tau|$ for $|\tau|>0$. 

     Pick a tubular neighbourhood $U_i$ of $Y_i$ in $X^{i-1}$ together with an associated diffeomorphism $f$ between $N_{Y_i|X^{i-1}}$ and $U_i$. Let $F:N_{Y_i|X^{i-1}}\times \CC\to \cX^i$ be defined by $F(\zeta,\tau)\coloneqq (f(\tau\zeta),\tau)$ for $\tau\neq 0$ while $F(\zeta,0)\coloneqq \zeta\in X^i$. Then $F$ is a diffeomorphism between $N_{Y_i|X^{i-1}}\times \CC$ and its image. Pick a volume form $\mathrm{d}V$ on $N_{Y_i|X^{i-1}}$. Since $\cT_i$ has smooth potentials on the complement of an analytic set not containing $X^i$ we can write for $\tau\neq 0$: 
     \[
     F_{\tau}^*(({\cT_i}_{|\cX^i_{\tau}})^n)=g_{\tau}\,\mathrm{d}V,
     \]
     for $\tau=0$: 
     \[
     T_i^n=g_0\,\mathrm{d}V,
     \]
     and $g_{\tau}$ will converge pointwise to $g_0$ almost everywhere. It thus follows from Fatou's Lemma that 
     \[
     \int_{X^i}T_i^n=\int_{N_{Y_i|X^{i-1}}}g_0\,\mathrm{d}V\leq \varliminf_{\tau\to 0}\int_{N_{Y_i|X^{i-1}}}g_{\tau}\,\mathrm{d}V\leq \varliminf_{\tau\to 0}\int_{\cX^i_{\tau}}\cT_i^n\leq \int_{\cX^i_1}\cT_i^n,
     \]
     thus proving our claim \eqref{eq:momclaim}.

    Now we note that since ${\cT_1}|_{\cX^1_1}\in \mathcal{Z}_+(X,\xi)$ we have that $\int_{\cX^1_1}\cT_1^n\leq \vol(\xi)$ by definition. By Proposition~\ref{prop:momcureq} we also have that 
    \[
    \vol_{\RR^n}(\Delta(T_n))=\frac{1}{n!}\int_{\PP^n}T_n^n,
    \]
    and combined with the inequalities \eqref{eq:momclaim} the proposition follows.
\end{proof}

Given a big class $\xi$ and a flag $Y_{\bullet}$, we can construct the Okounkov body $\Delta_{Y_{\bullet}}(\xi)$ and the moment body $\Delta_{Y_\bullet}^{\mu}(\xi)$. We will now show that they in fact coincide, proving Theorem~\ref{thm_momentbody}. 

\begin{theorem} \label{thm:mombodeq8}

Under the above assumptions, we have that 
\begin{equation}\label{eq:Okounkovbodieseq}
\Delta^{\mu}_{Y_{\bullet}}(\xi)=\Delta_{Y_{\bullet}}(\xi).    
\end{equation}
\end{theorem}

\begin{proof}
Pick $T_0\in\Omega(X,\xi)$ and write $\nu_{Y_{\bullet}}(T_0)=(\nu_1,\dots,\nu_n)$. Pick also a smooth real closed $(1,1)$-form $\theta$ in the class $\xi-\nu_1[Y_1]$.
We can thus write  $T_0=\theta+\nu_1[Y_1]+\ddc\phi$, where $\phi\in \PSH(X,\theta)$ restricts to a $\theta|_{Y_1}$-psh function on $Y_1$. Let $D_1$ denote the strict transform of $Y_1\times \PP^1$ in $\cX^1$,  and let $\cT_1\coloneqq \pi^*_X\theta+\nu_1[D_1]+\ddc \pi_X^*\phi$. We see that $\cT_1$ is a closed positive $S^1$-invariant current on $\cX^1$ which restricts to $T_0$ on $\cX^1_1$, and which has a continuous potential away from an analytic set not containing $\cX^1_1$ or $X^1$.

Note that $T_1=\nu_1[Y_1]+\pi^*_{Y_1}\theta|_{Y_1}+\ddc\pi_{Y_1}^*(\phi|_{Y_1})$, where $\pi_{Y_1}$ denotes the projection from $X^1=\PP(N_{Y_1|X}\oplus \CC)$ to the zero section identified with $Y_1$. We now see that we can write $T_1=\theta_1+\nu_1[Y_1]+\nu_2[W_2]+\ddc\phi_1$, where $W_2\coloneqq \pi_{Y_1}^{-1}(Y_2)$, $\theta_1$ and $\phi_1$ are $S^1$-invariant and $\phi_1$ restricts to a ${\theta_1}|_{Y_2}$-psh function on $Y_2$.

Let $D_1$ and $D_2$ denote the strict transforms of $Y_1\times \PP^1$ and $W_2\times \PP^1$ in $\cX^2$ respectively. Let $\cT_2\coloneqq \pi^*_{X^1}\theta_1|_{Y_2}+\nu_1[D_1]+\nu_2[D_2]+\ddc \pi_{X^1}^*(\phi_1|_{Y_2})$. We see that $\cT_2$ is a closed positive $(S^1)^2$-invariant current on $\cX^1$ which restricts to $T_1$ on $\cX^2_1$, and which has a continuous potential away from an analytic set not containing $\cX^2_1$ or $X^2$. 

Note that $T_2=\nu_1[Y_1]+\nu_2[W_2]+\pi^*_{Y_2}{\theta_1}|_{Y_2}+\ddc \pi_{Y_2}^*(\phi_1|_{Y_2})$, where $\pi_{Y_2}$ denotes the projection from $X^2$ to $Y_2$. We now see that we can write $T_2=\theta_2+\nu_1[Y_1]+\nu_2[W_2]+\nu_3[W_3]+\ddc\phi_2$, where $W_3\coloneqq \pi_{Y_2}^{-1}(Y_3)$, $\theta_2$ and $\phi_2$ are $(S^1)^2$-invariant and $\phi_2$ restricts to a ${\theta_2}|_{Y_3}$-psh function on $Y_3$.

Continuing this way we get a toric degeneration $\cT_{\bullet}$ of $(X,\xi)$ such that $T_n=\nu_1[Y_1]+\nu_2[W_2]+\dots+\nu_n[W_n]$. The divisors $Y_1,W_2,\dots,W_n$ are toric divisors and in the standard toric coordinates $z_i$ we have that $Y_1=\{z_1=0\}, W_2=\{z_1=z_2=0\}$, and so on. Thus, on $\CC^n$, $T_n$ has the toric potential $\psi=\nu_1\log|z_1|^2+\dots+\nu_n\log|z_n|^2$, and $u_{\psi}(x_1,\dots,x_n)=\nu_1 x_1+\dots+\nu_n x_n$. We thus see that $\Delta(T_n)=\{(\nu_1,\dots,\nu_n)\}=\{\nu(T_0)\}$, hence $\Delta_{Y_{\bullet}}(\xi)\subseteq \Delta^{\mu}_{Y_{\bullet}}(\xi).$
In particular,
\[
\vol_{\mathbb{R}^n}(\Delta_{Y_{\bullet}}(\xi))\leq \vol_{\mathbb{R}^n}(\Delta^{\mu}_{Y_{\bullet}}(\xi))\leq \frac{1}{n!}\vol(\xi)=\vol_{\mathbb{R}^n}(\Delta_{Y_{\bullet}}(\xi)),
\]
where the second inequality follows from Proposition~\ref{prop:volmuOkounkovupperbound} and the equality follows from Theorem~\ref{thm:volumeidentity}. It follows that all inequalities are actual equalities. Since $\Delta_{Y_{\bullet}}(\xi)$ has positive volume, we conclude \eqref{eq:Okounkovbodieseq}.
\end{proof}

\begin{corollary} \label{cor:momapprox}
    For any toric degeneration $(\PP^n,T_n)$ of $(X,\xi)$ we have that 
    \[
    \Delta(T_n)\subseteq \Delta(\xi).
    \]
    Furthermore, for any $\epsilon>0$ we can find a toric degeneration $(\PP^n,T_n)$ of $(X,\xi)$ such that the Hausdorff distance between $\Delta(T_n)$ and $\Delta(\xi)$ is less than $\epsilon$. 
\end{corollary}

\begin{proof}
    Since by definition $\Delta(T_n)\subseteq \Delta^{\mu}(\xi)$, the first claim follows from Theorem~\ref{thm:mombodeq8}. 
    
    That the moment body $\Delta^{\mu}(\xi)$ can be approximated arbitrarily well (with respect to the Hausdorff distance) by moment bodies $\Delta(T_n)$ of toric degenerations $(\PP^n,T_n)$ follows from Proposition~\ref{prop:momunion}, so combined with Theorem~\ref{thm:mombodeq8} we get the second claim. 
\end{proof}

Corollary~\ref{cor:momapprox} provides a strong link between transcendental Okounkov bodies and toric degenerations, which could be useful in further investigations, for example studying the transcendental Chebyshev transform.

\bigskip

\noindent {\sc Tam\'as Darvas, University of Maryland}\\
{\tt tdarvas@umd.edu}\vspace{0.1in}

\noindent {\sc Rémi Reboulet, David Witt Nystr\"om,  Chalmers University of Technology and University of Gothenburg}\\
{\tt reboulet@chalmers.se, wittnyst@chalmers.se}\vspace{0.1in}

\noindent {\sc Mingchen Xia, Chalmers Tekniska Högskola and Institute of Geometry and Physics, USTC}\\
{\tt xiamingchen2008@gmail.com}\vspace{0.1in}

\noindent {\sc Kewei Zhang,  Beijing Normal University}\\
{\tt kwzhang@bnu.edu.cn}


\begin{thebibliography}{1}

\bibitem{AM69} M. F. Atiyah and I. G. Macdonald, 
{\it Introduction to commutative algebra}, 
Addison-Wesley Publishing Co., 1969.

\bibitem{And13} D. Anderson, 
{\it Okounkov bodies and toric degenerations}, 
Math. Ann. {\bf 356}.3 (2013), pp. 1183–1202.

\bibitem{BC11} S. Boucksom and H. Chen, 
{\it Okounkov bodies of filtered linear series}, 
Compos. Math. {\bf 147}.4 (2011), pp. 1205–1229.

\bibitem{BDPP13} S. Boucksom, J.-P. Demailly, M. Păun, and T. Peternell, 
{\it The pseudo-effective cone of a compact Kähler manifold and varieties of negative Kodaira dimension}, 
J. Algebraic Geom. {\bf 22}.2 (2013), pp. 201–248.

\bibitem{BEGZ10} S. Boucksom, P. Eyssidieux, V. Guedj, and A. Zeriahi, 
{\it Monge-Ampère equations in big cohomology classes}, 
Acta Math. {\bf 205}.2 (2010), pp. 199–262.

\bibitem{BFJ08} S. Boucksom, C. Favre, and M. Jonsson, 
{\it Valuations and plurisubharmonic singularities}, 
Publ. Res. Inst. Math. Sci. {\bf 44}.2 (2008), pp. 449–494.

\bibitem{BFJ09} S. Boucksom, C. Favre, and M. Jonsson, 
{\it Differentiability of volumes of divisors and a problem of Teissier}, 
J. Algebraic Geom. {\bf 18}.2 (2009), pp. 279–308.

\bibitem{BJ20} H. Blum and M. Jonsson, 
{\it Thresholds, valuations, and K-stability}, 
Adv. Math. {\bf 365} (2020), pp. 107062, 57.

\bibitem{BKMS15} S. Boucksom, A. Küronya, C. Maclean, and T. Szemberg, 
{\it Vanishing sequences and Okounkov bodies}, 
Math. Ann. {\bf 361}.3-4 (2015), pp. 811–834.

\bibitem{BM97} E. Bierstone and P. D. Milman, 
{\it Canonical desingularization in characteristic zero by blowing up the maximum strata of a local invariant}, 
Invent. Math. {\bf 128}.2 (1997), pp. 207–302.

\bibitem{Bou02a} S. Boucksom, 
{\it Cônes positifs des variétés complexes compactes}, 
Université Joseph-Fourier-Grenoble I, 2002.

\bibitem{Bou02b} S. Boucksom, 
{\it On the volume of a line bundle}, 
Internat. J. Math. {\bf 13}.10 (2002), pp. 1043–1063.

\bibitem{Bou04} S. Boucksom, 
{\it Divisorial Zariski decompositions on compact complex manifolds}, 
Ann. Sci. École Norm. Sup. (4) {\bf 37}.1 (2004), pp. 45–76.

\bibitem{CdS19} A. Cannas da Silva, 
{\it An invitation to symplectic toric manifolds}, 
Bol. Soc. Port. Mat. {\bf 77} (2019), pp. 119–132.

\bibitem{CFKL17} C. Ciliberto, M. Farnik, A. Küronya, V. Lozovanu, J. Roé, and C. Shramov, 
{\it Newton–Okounkov bodies sprouting on the valuative tree}, 
Rendiconti del Circolo Matematico di Palermo Series 2 {\bf 66}.2 (2017), pp. 161–194.

\bibitem{CGSZ19} D. Coman, V. Guedj, S. Sahin, and A. Zeriahi, 
{\it Toric pluripotential theory}, 
Ann. Polon. Math. {\bf 123}.1 (2019), pp. 215–242.

\bibitem{CGZ13} D. Coman, V. Guedj, and A. Zeriahi, 
{\it Extension of plurisubharmonic functions with growth control}, 
J. Reine Angew. Math. {\bf 676} (2013), pp. 33–49.

\bibitem{CHPW18} S. R. Choi, Y. Hyun, J. Park, and J. Won, 
{\it Okounkov bodies associated to pseudoeffective divisors}, 
J. Lond. Math. Soc. (2) {\bf 97}.2 (2018), pp. 170–195.

\bibitem{CT14} T. C. Collins and V. Tosatti, 
{\it An extension theorem for Kähler currents with analytic singularities}, 
Ann. Fac. Sci. Toulouse Math. (6) {\bf 23}.4 (2014).

\bibitem{CT22} T. C. Collins and V. Tosatti, 
{\it Restricted volumes on Kähler manifolds}, 
Ann. Fac. Sci. Toulouse Math. (6) {\bf 31}.3 (2022), pp. 907–947.

\bibitem{Dem12a} J.-P. Demailly, 
{\it Analytic methods in algebraic geometry}, 
Vol. 1, Surveys of Modern Mathematics, International Press, Higher Education Press, Beijing, 2012.

\bibitem{Dem12b} J.-P. Demailly, 
{\it Complex Analytic and Differential Geometry}, 
Available on personal website, 2012.

\bibitem{Dem15} J.-P. Demailly, 
{\it On the cohomology of pseudoeffective line bundles}, 
Complex geometry and dynamics, Vol. 10, Abel Symp., Springer, Cham, 2015, pp. 51–99.


\bibitem{Dem92} J.-P. Demailly,
{\it Regularization of closed positive currents and intersection theory},
J. Algebraic Geom. {\bf 1}.3 (1992), 361–409.

\bibitem{Den17} Y. Deng,
{\it Transcendental Morse inequality and generalized Okounkov bodies},
Algebr. Geom. {\bf 4}.2 (2017), 177–202.

\bibitem{DGZ16} S. Dinew, V. Guedj, and A. Zeriahi,
{\it Open problems in pluripotential theory},
Complex Var. Elliptic Equ. {\bf 61}.7 (2016), 902–930.

\bibitem{DP04} J.-P. Demailly and M. Păun,
{\it Numerical characterization of the Kähler cone of a compact Kähler manifold},
Ann. of Math. (2) {\bf 159}.3 (2004), 1247–1274.

\bibitem{DR17} R. Dervan and J. Ross,
{\it K-stability for Kähler manifolds},
Math. Res. Lett. {\bf 24}.3 (2017), 689–739.

\bibitem{Hir75} H. Hironaka,
{\it Flattening theorem in complex-analytic geometry},
Amer. J. Math. {\bf 97} (1975), 503–547.

\bibitem{His12} T. Hisamoto,
{\it Restricted Bergman kernel asymptotics},
Trans. Amer. Math. Soc. {\bf 364}.7 (2012), 3585–3607.

\bibitem{HK15} M. Harada and K. Kaveh,
{\it Integrable systems, toric degenerations and Okounkov bodies},
Invent. Math. {\bf 202}.3 (2015), 927–985.

\bibitem{JL23} C. Jiang and Z. Li,
{\it Algebraic reverse Khovanskii-Teissier inequality via Okounkov bodies},
Math. Z. {\bf 305}.2 (2023), Paper No. 26, 14.

\bibitem{Kav19} K. Kaveh,
{\it Toric degenerations and symplectic geometry of smooth projective varieties},
J. Lond. Math. Soc. {\bf 99}.2 (2019), 377–402.

\bibitem{KK12} K. Kaveh and A. G. Khovanskii,
{\it Newton-Okounkov bodies, semigroups of integral points, graded algebras and intersection theory},
Ann. of Math. (2) {\bf 176}.2 (2012), 925–978.

\bibitem{LM09} R. Lazarsfeld and M. Mustaţă,
{\it Convex bodies associated to linear series},
Ann. Sci. Éc. Norm. Supér. (4) {\bf 42}.5 (2009), 783–835.

\bibitem{LX18} C. Li and C. Xu,
{\it Stability of valuations: higher rational rank},
Peking Math. J. {\bf 1}.1 (2018), 1–79.

\bibitem{Mat13} S. Matsumura,
{\it Restricted volumes and divisorial Zariski decompositions},
Amer. J. Math. {\bf 135}.3 (2013), 637–662.

\bibitem{Mur20} T. Murata,
{\it Toric degenerations of projective varieties with an application to equivariant Hilbert functions},
Thesis (Ph.D.)–University of Pittsburgh, ProQuest LLC, Ann Arbor, MI, 2020.

\bibitem{Nak87} N. Nakayama,
{\it The lower semicontinuity of the plurigenera of complex varieties},
Algebraic geometry, Sendai, 1985, Vol. 10, Adv. Stud. Pure Math., North-Holland, Amsterdam, 1987, 551–590.

\bibitem{NWZ24} J. Ning, Z. Wang, and X. Zhou,
{\it On the extension of Kähler currents on compact Kähler manifolds: holomorphic retraction case},
Ann. Fac. Sci. Toulouse Math. (6) (2024), 183–195.

\bibitem{Oko96} A. Okounkov,
{\it Brunn-Minkowski inequality for multiplicities},
Invent. Math. {\bf 125}.3 (1996), 405–411.

\bibitem{RW17} J. Ross and D. Witt Nyström,
{\it Envelopes of positive metrics with prescribed singularities},
Ann. Fac. Sci. Toulouse Math. (6) {\bf 26}.3 (2017), 687–728.

\bibitem{SD18} Z. Sjöström Dyrefelt,
{\it K-semistability of cscK manifolds with transcendental cohomology class},
J. Geom. Anal. {\bf 28}.4 (2018), 2927–2960.

\bibitem{Siu74} Y. T. Siu,
{\it Analyticity of sets associated to Lelong numbers and the extension of closed positive currents},
Invent. Math. {\bf 27} (1974), 53–156.

\bibitem{Vu23} D.-V. Vu,
{\it Derivative of volumes of big cohomology classes},
2023, arXiv:2307.15909 [math.AG].

\bibitem{Wlo09} J. Włodarczyk,
{\it Resolution of singularities of analytic spaces},
Proceedings of Göokova Geometry-Topology Conference 2008, Göokova Geometry/Topology Conference (GGT), 2009, 31–63.

\bibitem{WN14} D. Witt Nyström,
{\it Transforming metrics on a line bundle to the Okounkov body},
Ann. Sci. Éc. Norm. Supér. (4) {\bf 47}.6 (2014), 1111–1161.



\bibitem{WN18} D. Witt Nyström,
{\it Canonical growth conditions associated to ample line bundles},
Duke Math. J. {\bf 167}.3 (2018), 449–495.

\bibitem{WN19a} D. Witt Nyström,
{\it Duality between the pseudoeffective and the movable cone on a projective manifold},
J. Amer. Math. Soc. {\bf 32}.3 (2019), 675–689.
With an appendix by Sébastien Boucksom.

\bibitem{WN19b} D. Witt Nyström,
{\it Monotonicity of non-pluripolar Monge-Ampère masses},
Indiana Univ. Math. J. {\bf 68}.2 (2019), 579–591.

\bibitem{WN21} D. Witt Nyström,
{\it Deformations of Kähler manifolds to normal bundles and restricted volumes of big classes},
J. Differential Geom. (to appear) (2021). arXiv:2103.03660 [math.AG].

\bibitem{Xia21} M. Xia,
{\it Partial Okounkov bodies and Duistermaat–Heckman measures of non-Archimedean metrics},
Geom. Topol. (to appear) (2021). arXiv:2112.04290 [math.AG].

\end{thebibliography}
\end{document}